\documentclass[hidelinks,11 pt]{article} 
\usepackage[top=22mm, bottom=22mm, left=22mm, right=22mm]{geometry}
\usepackage{rotating}
\usepackage{amsmath,amsfonts,amssymb,amsthm,enumerate,hyperref,multicol, graphicx, array, caption, subcaption, color,makecell}
\hypersetup{colorlinks,citecolor=blue, linkcolor=blue}
\usepackage{booktabs,siunitx}
\usepackage[dvipsnames]{xcolor}
\usepackage{multirow}
\usepackage{booktabs}
\usepackage{multirow}
\usepackage{array}
\usepackage{mathtools}
\usepackage{bigstrut}
\usepackage{float}
\usepackage{tabularx}
\usepackage{pgfplots}
\allowdisplaybreaks
\usepackage{titlesec}
\usepackage{enumitem}
\usepackage{authblk}

\usepackage{lipsum}
\usepackage{tikz}
\usetikzlibrary{shapes, arrows.meta, positioning}
\definecolor{orcidlogocol}{HTML}{A6CE39}
\usepackage{orcidlink}
\usepackage{amsthm}
\numberwithin{equation}{section}

\newtheorem{theorem}{Theorem}[section]
\newtheorem{lemma}[theorem]{Lemma}
\newtheorem{proposition}[theorem]{Proposition}

\theoremstyle{definition}  
\newtheorem{definition}[theorem]{Definition}

\newtheorem{remark}[theorem]{Remark}

\usepackage{orcidlink}
\numberwithin{equation}{section}
\numberwithin{theorem}{section}
\pgfplotsset{compat=1.18}
\begin{document}
{\LARGE{\title{{\bf Asymptotic Properties of a Forward-Backward-Forward Differential Equation and Its Discrete Version for Solving Quasimonotone Variational Inequalities}}}}

\author[1]{Yeyu Zhang$^{\orcidlink{0000-0001-7974-9636}}$}
\author[1]{Hongwei Liu\thanks{Corresponding author: hwliuxidian@163.com}}

\affil[1]{\small{School of Mathematics and Statistics, Xidian University,
\protect\\ Shaanxi, 710126, China\protect\\
{zyy835858276@163.com}, {hwliuxidian@163.com}}}

\date{}
\maketitle

\begin{abstract}
This paper investigates the asymptotic behavior of a forward-backward-forward (FBF) type differential equation and its discrete counterpart for solving quasimonotone variational inequalities (VIs). Building on recent continuous-time dynamical system frameworks for VIs, we extend these methods to accommodate quasimonotone operators. We establish weak and strong convergence under significantly relaxed conditions, without requiring strong pseudomonotonicity or sequential weak-to-weak continuity. Additionally, we prove ergodic convergence of the continuous trajectories, offering further insight into the long-term stability of the system. In the discrete setting, we propose a novel Bregman-type algorithm that incorporates a nonmonotone adaptive step-size rule based on the golden ratio technique. A key contribution of this work is demonstrating that the proposed method ensures strong convergence under the assumption of uniform continuity of the operator, thereby relaxing the standard Lipschitz continuity requirement prevalent in existing methods. Numerical experiments, including infinite-dimensional and non-Lipschitz cases, are presented to illustrate the improved convergence and broader applicability of the proposed approach.

\end{abstract}

\vspace{2mm}
\noindent
\textbf{Keywords:} Forward-backward-forward dynamical system
; Bregman algorithm; Quasimonotone variational inequality; Non-Euclidean adaptive method; Adaptive stepsize.\\[1mm]
\noindent
\textbf{2020 Mathematics Subject Classification:} 34G25; 47J20; 49J40; 65K15.

\section{Introduction}
\quad Variational inequalities (VIs) were originally introduced by G. Stampacchia in the early 1960s \cite{Stampacchia1964} as a tool for analyzing elliptic partial differential equations. Since their inception, VIs have evolved into a fundamental framework in optimization theory and applied mathematics.  
They naturally emerge across a broad spectrum of applications, such as equilibrium modeling in economics and game theory \cite{FacchineiPang2003}, traffic network analysis \cite{Nagurney1993}, studies of electricity markets \cite{Bertsekas1999}, and machine learning problems involving monotone operator theory \cite{BauschkeCombettes2011}.
\vskip 2mm
The classical formulation of a variational inequality problem in a Hilbert space \(\mathcal{H}\) reads as follows:
\begin{equation}\label{1.1}
    \text{Find } x^* \in K \subset\mathcal{H} \text{ such that } \langle F(x^*), y - x^* \rangle \geq 0, \quad \forall y \in K,
\end{equation}
where \( F: \mathcal{H} \to\mathcal{H}\) is a given continuous operator, \( K \subset\mathcal{H} \) is a nonempty closed convex subset, and \( \langle \cdot, \cdot \rangle \) denotes the inner product in \( \mathcal{H} \). This problem is referred to as \( VI(F, K) \), and we denote its solution set by \( \Omega \).
\vskip 2mm
Classical approaches to variational inequalities mainly rely on discrete iterative schemes. Among them, the projection gradient (PG) method \cite{a6} is one of the most widely used algorithms. To ensure the convergence of the PG method, it is typically required that the operator \( F \) be strongly monotone and cocoercive, along with being \( L \)-Lipschitz continuous over \( K \). As demonstrated in \cite{FacchineiPang2003}, if \( F \) is merely monotone and \( L \)-Lipschitz continuous, the sequence generated by the PG method may fail to converge.
\vskip 2mm
To address the limitation of requiring strong monotonicity, various projection-based methods have been proposed for solving monotone VIs; see, e.g., \cite{a7,a8,a9,a10,a11,a12}. Among these, a notable algorithm is the extragradient method developed by Korpelevich \cite{a11}, which is designed for finite-dimensional spaces under the assumptions that \( F \) is monotone and Lipschitz continuous. The extragradient method involves two projections onto the feasible set \( K \) per iteration. However, when \( K \) has a complex structure, evaluating these projections can become challenging or may lack an explicit closed-form expression.
To address this drawback, Censor et al.\ \cite{a12} introduced the subgradient extragradient method for solving monotone and Lipschitz continuous variational inequalities. In contrast to the classical extragradient method, which requires two projections onto the feasible set \( K \), their approach replaces the second projection with a projection onto a suitably constructed half-space \( T_n \). This modification not only preserves the convergence properties but also significantly reduces the computational burden when projections onto \( K \) are difficult to compute. Another important development is the forward-backward-forward (FBF) method proposed by Tseng in \cite{a13}. In the FBF method, the second projection step in the extragradient algorithm is replaced by an explicit forward computation involving evaluations of the operator \( F \). Similar to the extragradient framework, the FBF method also assumes that \( F \) is monotone and Lipschitz continuous. The advantage of the FBF method lies in its ability to maintain the convergence guarantees while potentially simplifying the implementation compared to methods requiring multiple projections.
In recent years, the development of algorithms for nonmonotone variational inequalities has gained considerable attention, particularly due to their applications in machine learning, see \cite{a14,a15,a16,a17,a31,a32}. Consequently, the above mentioned method  were later used to solve VIs with the cost operators being pseudomonotone and quasimonotone. 
\vskip 2mm 
In addition to discrete iterative schemes, continuous dynamical systems have been developed as powerful tools to analyze and solve variational inequalities. Motivated by the pioneering work of Bruck \cite{Bruck1975} on the asymptotic behavior of nonlinear semigroups associated with maximal monotone operators, continuous-time approaches offer a complementary perspective to discrete algorithms.
\vskip 2mm 
In particular, a projected dynamical system has been proposed and extensively studied for solving monotone variational inequalities. The system is given by
\begin{equation}\label{1.2}
    \dot{x}(t) + x(t) - P_K\left(x(t) - \lambda F(x(t))\right) = 0, \quad t \geq 0,
\end{equation}
where \( \lambda > 0 \) is a fixed parameter, \( P_K \) denotes the projection onto the closed convex set \( K \), and \( F: \mathcal{H} \to \mathcal{H}\) is a monotone and Lipschitz continuous operator. This continuous formulation can be interpreted as the continuous limit of the classical projection gradient method and has been systematically investigated in \cite{Alvarez2000, AttouchCzarnecki2011}. The projected dynamical system \eqref{1.2} exhibits several desirable properties under suitable assumptions on \( F \) and \( K \), such as the existence and uniqueness of global solutions, boundedness of trajectories, and convergence of the trajectories to a solution of the variational inequality problem \( VI(F, K) \). These properties make continuous dynamical systems a valuable framework for both theoretical analysis and algorithmic development.
\vskip 2mm 
Recently, Bot et al.\cite{a22} adopted a continuous-time perspective to investigate the solution set of \( VI(F, K) \) by studying the trajectories generated by a forward-backward-forward-type dynamical system, specifically designed for solving pseudomonotone variational inequalities:
\begin{equation} \label{4}
    \begin{cases}
y(t) = P_K(x(t) - \lambda F(x(t))) \\ 
\dot{x}(t) + x(t) = y(t) + \lambda \big[F(x(t)) - F(y(t))\big] \\ 
x(0) = x_0, 
\end{cases}
\end{equation}
The formulation of \eqref{4} traces back to the work of Banert and Bot \cite{a23}, where the continuous counterpart of Tseng’s algorithm was investigated in the broader framework of monotone inclusion problems. The existence and uniqueness of the trajectory \( x(t) \) generated by \eqref{4} were established therein as a consequence of the global Cauchy--Lipschitz theorem, utilizing the Lipschitz continuity of the operator \( F \).    
Under the assumptions that the operator $F$ is pseudo-monotone, sequentially weak-to-weak continuous, and $L$-Lipschitz continuous with constant $L>0$, and that the parameter $\lambda$ satisfies $0<\lambda<\frac{1}{L}$, it has been shown that the trajectories $x(t)$ and $y(t)$ generated by \eqref{4} converge weakly to a solution of $VI(F,K)$ as $t\to\infty$ in real Hilbert spaces. Furthermore, if $F$ is additionally strongly pseudo-monotone on $K$, then the trajectories exhibit exponential convergence to the unique solution of $VI(F,K)$. 
In addition, an explicit time discretization of the continuous system leads to Tseng's forward-backward-forward algorithm with relaxation parameters. It is also established that the sequence $\{x_n\}$ generated by the discrete scheme converges weakly to a solution of $VI(F,K)$, and linear convergence is achieved under the strong pseudo-monotonicity condition.
\vskip 2mm 
What we need to emphasize is that research on pseudomonotone VIs rely on Minty conditions, that is, let \( \Omega_{D} \) represent the solution set of the following problem, which seeks to find \( x^* \in K \) such that
\begin{equation}
    \langle F(y), y - x^* \rangle \geq 0, \quad \forall y \in K.
\end{equation}
Clearly, \( \Omega_D \) is a closed and convex set, although it may be empty. When \( F \) is a continuous mapping and \( K \) is convex, we have the inclusion \( \Omega_D \subset \Omega \). Moreover, the equality \( \Omega_D = \Omega \) holds if \( F \) is both continuous and pseudomonotone. However, it is crucial to emphasize that the reverse inclusion \( \Omega \subset \Omega_D \) does not generally hold when \( F \) is quasimonotone and continuous. In other words, a solution \( x^* \) of the variational inequality problem \(VI(F, K) \) is not necessarily contained in \( \Omega_D \). This discrepancy introduces a fundamental challenge: in the general setting where the solution \( x^* \in \Omega \), how can one design effective algorithms and identify suitable conditions to ensure convergence or feasibility with respect to \( \Omega_D \). 
\vskip 2mm 
Throughout this paper, we further investigate forward-backward-forward (FBF) type dynamical systems for solving variational inequality problems. Compared with the foundational results of Bot et al.~\cite{a23}, our work offers a broader framework under relaxed assumptions, leading to enhanced convergence properties and wider applicability. The specific contributions of this paper are as follows:

\vskip 2mm
$\bullet$ We extend the FBF-type dynamical system \eqref{4} to address quasimonotone variational inequalities. Beyond the results of Bot et al.~\cite{a23}, we establish not only the weak convergence of the trajectories $x(t)$ but also their strong convergence under weaker and more practical assumptions, notably without requiring strong pseudomonotonicity or sequential weak-to-weak continuity of $F$. Building on the strong convergence results, we further prove the ergodic (time-averaged) convergence of the trajectories, providing a refined understanding of the system's asymptotic behavior.

\vskip 2mm
$\bullet$ We investigate the time discretization of the dynamical system \eqref{4} and study the iterative approximation of solutions to quasimonotone variational inequalities in reflexive Banach spaces, extending beyond the traditional Hilbert space setting. By utilizing the Bregman distance, Bregman projections, and a golden ratio-based technique (see \cite{a26}), we establish strong convergence of the proposed algorithm under suitable conditions. Furthermore, we introduce a nonmonotone adaptive step-size strategy that eliminates the need for Lipschitz continuity of the operator $F$, relaxing it to uniform continuity, thereby significantly weakening the assumptions typically required and broadening the algorithm's applicability.

\vskip 2mm
$\bullet$ To demonstrate the practical efficiency and robustness of the proposed algorithm, we conduct numerical experiments on representative quasimonotone variational inequality problems, including examples in infinite-dimensional Hilbert spaces and traffic equilibrium models. The results show that our method converges reliably under weaker assumptions and outperforms several benchmark algorithms in terms of iteration count and solution accuracy.

\vskip 2mm 
The rest of the paper is organized as follows. In Section~\ref{sec:prelim}, we introduce the necessary preliminaries and notations. Section~\ref{sec:dynsys} presents the asymptotic analysis of a forward-backward-forward-type dynamical system for quasimonotone variational inequalities. In Section~\ref{sec:algorithm}, we propose a discrete forward-backward-forward algorithm incorporating the golden ratio and relaxation parameters, along with its convergence analysis. Section~\ref{sec:experiments} reports numerical experiments that illustrate the effectiveness and advantages of the proposed method.

\section{Preliminaries}\label{sec:prelim}
\begin{definition}\label{def:monotonicity}
Let \( \mathcal{H} \) be a Hilbert space and let \( F : \mathcal{H} \to \mathcal{H} \) be an operator. Then:

\begin{itemize}
    \item[(i)] \( F \) is said to be \textbf{monotone} on \( \mathcal{H} \) if
    \[
    \langle F(\xi) - F(\eta), \xi - \eta \rangle \geq 0, \quad \forall \xi, \eta \in \mathcal{H}.
    \]

    \item[(ii)] \( F \) is said to be \textbf{strongly monotone} on \( \mathcal{H} \) with parameter \( \mu > 0 \) if
    \[
    \langle F(\xi) - F(\eta), \xi - \eta \rangle \geq \mu \|\xi - \eta\|^2, \quad \forall \xi, \eta \in \mathcal{H}.
    \]

    \item[(iii)] \( F \) is called \textbf{pseudomonotone} on a subset \( K \subseteq \mathcal{H} \) if, for all \( \xi, \eta \in K \),
    \[
    \langle F(\xi), \eta - \xi \rangle \geq 0 \ \Rightarrow\ \langle F(\eta), \eta - \xi \rangle \geq 0.
    \]

    \item[(iv)] \( F \) is said to be \textbf{strongly pseudomonotone} on \( K \subseteq \mathcal{H} \) with parameter \( \mu > 0 \) if, for all \( \xi, \eta \in K \),
    \[
    \langle F(\xi), \eta - \xi \rangle \geq 0 \ \Rightarrow\ \langle F(\eta), \eta - \xi \rangle \geq \mu \|\eta - \xi\|^2.
    \]

    \item[(v)] \( F \) is called \textbf{quasimonotone} on \( K \subseteq \mathcal{H} \) if, for all \( \xi, \eta \in K \),
    \[
    \langle F(\xi), \eta - \xi \rangle > 0 \ \Rightarrow\ \langle F(\eta), \eta - \xi \rangle \geq 0.
    \]

    \item[(vi)] \( F \) is said to be \textbf{\( L \)-Lipschitz continuous} if there exists a constant \( L > 0 \) such that
    \[
    \|F(\xi) - F(\eta)\| \leq L \|\xi - \eta\|, \quad \forall \xi, \eta \in \mathcal{H}.
    \]

    \item[(vii)] \( F \) is said to be \textbf{uniformly continuous} if, for every \( \epsilon > 0 \), there exists a constant \( \delta = \delta(\epsilon) > 0 \) such that
    \[
    \|\xi - \eta\| < \delta \ \Rightarrow\ \|F(\xi) - F(\eta)\| < \epsilon, \quad \forall \xi, \eta \in \mathcal{H}.
    \]
\end{itemize}
\end{definition}

\begin{lemma}\cite{a27}\label{l2.2}
    If the operator $F:K\to E$ is uniformly continuous if and only if,  for all $\epsilon>0$, there exists a constant $M>0$, such that 
    \begin{equation*}
        \|Fx-Fy\|\leq M\|x-y\|+\epsilon, \forall x,y\in K.
    \end{equation*}
\end{lemma}

\begin{definition}
Let \( \mathcal{H} \) be a Hilbert space, and let \( K \subseteq \mathcal{H} \) be a nonempty, closed, and convex subset. The \textbf{metric projection} of a point \( \xi \in \mathcal{H} \) onto \( K \), denoted by \( P_K(\xi) \), is defined as the unique point \( \zeta \in K \) such that
\[
\|\xi - \zeta\| = \inf_{\eta \in K} \|\xi - \eta\|.
\]
The mapping \( P_K : \mathcal{H} \to K \) is called the \textbf{projection operator} onto \( K \).
\end{definition}

\begin{proposition}[The obtuse-angle property]
Let \( \mathcal{H} \) be a Hilbert space, and let \( K \subseteq \mathcal{H} \) be a nonempty, closed, and convex set. Then for any \( \xi \in \mathcal{H} \) and \( \zeta = P_K(\xi) \), the following inequality holds:
\[
\langle \xi - \zeta, \eta - \zeta \rangle \leq 0, \quad \forall \eta \in K.
\]
In other words, the vector from \( \xi \) to its projection \( \zeta \) forms an obtuse angle (or a right angle) with any vector in \( K \) emanating from \( \zeta \).
\end{proposition}
\begin{lemma}\cite{a28}
\label{lemma:asymptotic_decay}
Let \(1 \leq p < \infty\), \(1 \leq r < \infty\), and let \(A : [0, +\infty) \to [0, +\infty)\) be a locally absolutely continuous function such that \(A \in L^p([0, +\infty))\). Suppose that \(B : [0, +\infty) \to \mathbb{R}\) satisfies \(B \in L^r([0, +\infty))\), and assume that
\[
\frac{d}{dt} A(t) \leq B(t), \quad \text{for almost every } t \in [0, +\infty).
\]
Then it follows that
\[
\lim_{t \to +\infty} A(t) = 0.
\]
\end{lemma}
\vskip 2mm 
Let \( E \) be a real reflexive Banach space with dual \( E^* \). Let \( \Phi: E \to (-\infty, +\infty] \) be a proper, convex, and lower semicontinuous function that is Gâteaux differentiable on the interior of its effective domain, denoted by \( \operatorname{int}(\operatorname{dom} \Phi) \).

\vskip 2mm
The Fenchel conjugate of \( \Phi \), written as \( \Phi^*: E^* \to \mathbb{R} \), is defined via the supremum
\[
\Phi^*(u^*) = \sup_{u \in E} \left\{ \langle u, u^* \rangle - \Phi(u) \right\}.
\]

\vskip 2mm
For any \( x \in \operatorname{int}(\operatorname{dom} \Phi) \) and direction \( y \in E \), the directional derivative of \( \Phi \) at \( x \) along \( y \) is given by
\begin{equation}\label{dir_derivative}
\Phi^{\circ}(x, y) = \lim_{r \to 0^+} \frac{\Phi(x + r y) - \Phi(x)}{r}.
\end{equation}

\vskip 2mm
We say that \( \Phi \) is:
\begin{itemize}
  \item \textbf{Gâteaux differentiable} at \( x \in \operatorname{int}(\operatorname{dom} \Phi) \) if the limit in \eqref{dir_derivative} exists for all directions \( y \in E \). The gradient \( \nabla \Phi(x) \in E^* \) is then defined through the identity
$  \langle \nabla \Phi(x), y \rangle = \Phi^{\circ}(x, y), \quad \forall y \in E$.
  \item \textbf{Fréchet differentiable} at \( x \in \operatorname{int}(\operatorname{dom} \Phi) \) if the directional derivative \eqref{dir_derivative} is attained uniformly for all \( y \in E \) with \( \|y\| = 1 \) as \( r \to 0^+ \).
\end{itemize}

\vskip 2mm
The function \( \Phi: E \to (-\infty, +\infty] \) is referred to as a \textbf{Legendre function} if the following conditions hold:
\begin{enumerate}
  \item \( \Phi \) is Gâteaux differentiable, and \( \operatorname{int}(\operatorname{dom} \Phi) \neq \emptyset \), with $\operatorname{dom} \nabla \Phi = \operatorname{int}(\operatorname{dom} \Phi)$.
  
  \item The conjugate \( \Phi^* \) is also Gâteaux differentiable, and similarly,
 $ \operatorname{int}(\operatorname{dom} \Phi^*) \neq \emptyset, \quad \operatorname{dom} \nabla \Phi^* = \operatorname{int}(\operatorname{dom} \Phi^*)$.

\end{enumerate}

Under these conditions, the gradient mapping \( \nabla \Phi \) is bijective with its inverse given by the gradient of the conjugate:
\[
(\nabla \Phi)^{-1} = \nabla \Phi^*.
\]

\vskip 2mm
\begin{definition}
Let \( \Phi: E \to \mathbb{R} \cup \{+\infty\} \) be proper, convex, lower semicontinuous, and Gâteaux differentiable on \( \operatorname{int}(\operatorname{dom} \Phi) \). The associated \textbf{Bregman distance}(or Bregman divergence) \( D_\Phi: \operatorname{dom} \Phi \times \operatorname{int}(\operatorname{dom} \Phi) \to \mathbb{R} \) is defined as
\[
D_\Phi(x, y) = \Phi(x) - \Phi(y) - \langle \nabla \Phi(y), x - y \rangle.
\]
\end{definition}

\vskip 2mm
This Bregman divergence satisfies the identity
\begin{equation}\label{bregman_identity}
D_\Phi(x, y) + D_\Phi(y, z) - D_\Phi(x, z) = \langle \nabla \Phi(z) - \nabla \Phi(y), x - y \rangle,
\end{equation}
for all \( x \in \operatorname{dom} \Phi \) and \( y, z \in \operatorname{int}(\operatorname{dom} \Phi) \). This property is essential for the analysis of convergence in optimization algorithms. If the function $\Phi$ is strongly convex with the strong convexity constant $\rho>0$, then we have $  D_{\Psi}(x,y)\geq \frac{\rho}{2}\|x-y\|^{2}$. In addition, for any $x\in \operatorname{dom} \Phi$, $y,z,w\in \operatorname{int (dom} \Phi)$ and $a\in\mathbb{R}$, such that 
\begin{equation*}
    \nabla \Phi (y) = a\nabla \Phi (z)+(1-a)\nabla\Phi(w), 
\end{equation*}
we have 
\begin{equation}\label{2.3}
    D_{\Phi}(x,y) = a\left(D_{\Phi}(x,z)-D_{\Phi}(y,z)\right)+(1-a)\left(D_{\Phi}(x,w)-D_{\Phi}(y,w)\right). 
\end{equation}

\vskip 2mm
The function \( \Phi \) is said to be \textbf{strongly coercive} if
\[
\lim_{\|x\| \to \infty} \frac{\Phi(x)}{\|x\|} = +\infty.
\]
This condition ensures boundedness of sublevel sets, which is crucial for ensuring the existence of minimizers in variational problems.

\vskip 2mm
Given a closed convex subset \( K \subset E \), the \textbf{Bregman projection} of a point \( x \in \operatorname{int}(\operatorname{dom} \Phi) \) onto \( K \) with respect to \( \Phi \) is defined as the unique solution to
\[
\operatorname{proj}_K^\Phi(x) = \arg\min_{y \in K} D_\Phi(y, x).
\]
Moreover, \( z = \operatorname{proj}_K^\Phi(x) \) if and only if
\begin{equation}\label{proj_vi}
\langle \nabla \Phi(x) - \nabla \Phi(z), y - z \rangle \leq 0, \quad \forall y \in K.
\end{equation}
The following inequality further characterizes the projection
\begin{equation}\label{bregman_projection_ineq}
D_\Phi(y, \operatorname{proj}_K^\Phi(x)) + D_\Phi(\operatorname{proj}_K^\Phi(x), x) \leq D_\Phi(y, x), \quad \forall y \in K.
\end{equation}

\vskip 2mm
Assuming \( \Phi \) is a Legendre function, we define the following quantity
\[
V_\Phi(x, x^*) = \Phi(x) - \langle x, x^* \rangle + \Phi^*(x^*), \quad \forall x \in E, \, x^* \in E^*.
\]
This function satisfies
\[
V_\Phi(x, x^*) = D_\Phi(x, \nabla \Phi^*(x^*)).
\]
Consequently, the convexity of \( D_\Phi \) in its second variable yields the inequality
\begin{equation}\label{bregman_jensen}
D_\Phi\left(z, \nabla \Phi^*\left( \sum_{i=1}^N t_i \nabla \Phi(x_i) \right)\right) \leq \sum_{i=1}^N t_i D_\Phi(z, x_i),
\end{equation}
where \( \{x_i\} \subset E \), \( \{t_i\} \subset (0,1) \), and \( \sum_{i=1}^N t_i = 1 \).

\vskip 2mm
\begin{lemma}\label{l2.7}
Let $\{\Xi_k\}$ and $\{\Lambda_k\}$ be nonnegative sequences. If
\[
\Xi_{k+1} \leq \Xi_k - \Lambda_k + \theta_k \quad \text{for all } k > m,
\]
where $\{\theta_k\}$ is a nonnegative sequence such that $\sum_{k=m+1}^{\infty} \theta_k < \infty$, then the following hold:
\begin{itemize}
    \item[(i)] The sequence $\{\Xi_k\}$ converges;
    \item[(ii)] $\lim_{k \to \infty} \Lambda_k = 0$.
\end{itemize}

Furthermore, if we set $\theta_k := 0$, the same conclusions hold.
\end{lemma}
\begin{proof}
Since $\{\Xi_k\}$ and $\{\Lambda_k\}$ are nonnegative sequences and
\[
\Xi_{k+1} \leq \Xi_k - \Lambda_k + \theta_k \quad \text{for all } k > m,
\]
we sum both sides from $k = m+1$ to $k = N$ to obtain
\[
\Xi_{N+1} \leq \Xi_{m+1} - \sum_{k=m+1}^{N} \Lambda_k + \sum_{k=m+1}^{N} \theta_k.
\]
Rearranging terms yields
\[
\sum_{k=m+1}^{N} \Lambda_k \leq \Xi_{m+1} - \Xi_{N+1} + \sum_{k=m+1}^{N} \theta_k.
\]
Since $\Xi_k \geq 0$ and $\sum_{k=m+1}^{\infty} \theta_k < \infty$ by assumption, we conclude that $\sum_{k=m+1}^{\infty} \Lambda_k < \infty$,
which implies that $\Lambda_k \to 0$ as $k \to \infty$. This proves part (ii).
\vskip 2mm 
To prove part (i), observe that the inequality also implies
\[
\Xi_{k+1} \leq \Xi_k + \theta_k.
\]
Since the sequence $\{\theta_k\}$ is nonnegative and summable, i.e., $\sum_{k=m+1}^\infty \theta_k < \infty$, it follows that any potential increases in $\{\Xi_k\}$ are uniformly controlled by a convergent series. More precisely, for any integers $p > q > m$, we have
\[
\Xi_p \leq \Xi_q + \sum_{k=q}^{p-1} \theta_k \leq \Xi_q + \sum_{k=q}^{\infty} \theta_k,
\]
which implies that the sequence $\{\Xi_k\}$ is a Cauchy sequence. Therefore, it converges.
\vskip 2mm 
Finally, when $\theta_k := 0$ for all $k$, the same arguments apply, and both conclusions remain valid.
\end{proof}

\section{Asymptotic Analysis of a Forward-Backward-Forward-Type Dynamical System}\label{sec:dynsys}

\quad In this section, we investigate the asymptotic behavior of a forward-backward-forward-type (FBF-type) dynamical system \eqref{4} for solving quasimonotone variational inequality problems. 
 \vskip 2mm 
The analysis is conducted under the following assumptions:

\vskip 2mm
\textbf{(C1)}: The solution set \( \Omega_{D} \) is nonempty; that is, \( \Omega_{D} \neq \emptyset \).

\vskip 2mm
\textbf{(C2)}: The operator \( F:\mathcal{H} \to \mathcal{H} \) is quasimonotone and $L-$Lipschitz continuous.

\vskip 2mm
\textbf{(C3)}: The mapping \( F \) satisfies a sequential lower semicontinuity condition: whenever a sequence \( \{x_k\} \subset \mathcal{H} \) converges weakly to \( x \), i.e., \( x_k \rightharpoonup x \), it holds that
\[
\|F(x)\| \leq \liminf_{k \to \infty} \|F(x_k)\|.
\]

\vskip 2mm
\textbf{(C4)}: Let \( A := \{u \in K : F(u) = 0\} \setminus \Omega_D \). Then for every \( p \in \Omega_D \), there exists a constant \( \delta_p > 0 \) such that for any two distinct points \( y_1, y_2 \in A_p := \{p\} \cup A \), it holds that
\[
\|y_1 - y_2\| \geq 2\delta_p.
\]

\begin{remark}
 The existence and uniqueness of the trajectory \(x(t)\) generated by system \eqref{4} have been established by Banert and Bot \cite{a23}, as a consequence of the global Cauchy–Lipschitz theorem, leveraging the Lipschitz continuity of the operator \(F\).
\end{remark}
\begin{remark}
    Note that condition \textbf{(C3)} is strictly weaker than the sequential weak-to-weak continuity of the operator \(F\). A detailed justification and relevant examples can be found in recent works on solving quasimonotone variational inequality problems, see \cite{a24}. This relaxation significantly broadens the applicability of our method to a wider class of operators that may fail to satisfy weak sequential continuity.
\end{remark}
In what follows, we start our asymptotic analysis of FBF-type dynamical system for solving quasimonotone variational inequality problems. The following proposition is a natural extension of some results of Bot et al.\cite{a22}. 
\begin{proposition}\label{p3.4}
    Suppose that $\Omega_{D} \neq \emptyset$, $F$ is quasimonotone on $K$ and $L$-Lipschitz continuous, and $\lambda\in(0,\frac{1}{L})$. Then for any $x^l \in \Omega_{D}$, the following inequality holds:
    \begin{equation}
        \langle \dot{x}(t), x(t) - x^l \rangle \leq -\left(1 -\lambda L\right) \|x(t) - y(t)\|^2, \quad \forall t \in [0,+\infty).
    \end{equation}
    In addition, the function $t \mapsto \|x(t) - x^l\|^2$ is nonincreasing, and it holds that
    \begin{equation*}
        \int_0^{+\infty} \|x(t) - y(t)\|^2 \, dt < \infty, \quad \text{and} \quad \lim_{t\to\infty}\|x(t)-y(t)\|=0.
    \end{equation*}
\end{proposition}

\begin{proof}
    Since $y(t) = P_K(x(t) - \lambda F(x(t)))$, by the obtuse-angle property of the projection, we have
    \begin{equation*}
        \langle x(t) - \lambda F(x(t)) - y(t), y(t) - x^l \rangle \geq 0, \quad \forall t \in [0, +\infty).
    \end{equation*}
    Moreover, since $x^l \in \Omega_D$ and $y(t) \in K$, which implies 
    \begin{equation*}
        \langle F(y(t)), y(t) - x^l \rangle \geq 0.
    \end{equation*}
    Combining the above two inequalities yields
    \begin{equation*}
        \langle x(t) - y(t) - \lambda[F(x(t)) - F(y(t))], y(t) - x^l \rangle \geq 0.
    \end{equation*}
    Using the formulation of the dynamical system \eqref{4}, we obtain
    \begin{align*}
        &\langle x(t) - y(t) - \lambda[F(x(t)) - F(y(t))], y(t) - x(t) \rangle - \langle \dot{x}(t), x(t) - x^l \rangle \geq 0,
    \end{align*}
    which implies
    \begin{align*}
        \langle \dot{x}(t), x(t) - x^l \rangle &\leq \langle x(t) - y(t) - \lambda[F(x(t)) - F(y(t))], y(t) - x(t) \rangle \\
        &= -\|x(t) - y(t)\|^2 + \lambda \langle F(x(t)) - F(y(t)), x(t) - y(t) \rangle \\
        &\leq -\|x(t) - y(t)\|^2 + \lambda \|F(x(t)) - F(y(t))\| \cdot \|x(t) - y(t)\| \\
        &\leq -\|x(t) - y(t)\|^2 + \lambda L \|x(t) - y(t)\|^2 \\
        &= -\left(1 -\lambda L\right) \|x(t) - y(t)\|^2. 
    \end{align*}
\vskip 2mm 
    Therefore, for all $t \in [0, +\infty)$,
    \begin{equation*}
        \frac{1}{2} \frac{d}{dt} \|x(t) - x^l\|^2 = \langle \dot{x}(t), x(t) - x^l \rangle \leq -\left(1 - \lambda L\right) \|x(t) - y(t)\|^2 \leq 0,
    \end{equation*}
    which shows that the function $t \mapsto \|x(t) - x^l\|^2$ is nonincreasing.
\vskip 2mm 
    Let $S > 0$. Integrating the above inequality over $[0, S]$, we obtain
    \begin{align*}
        \int_0^S \left(1 - \lambda L\right) \|x(t) - y(t)\|^2 \, dt &\leq \int_0^S -\frac{d}{dt} \left(\frac{1}{2} \|x(t) - x^l\|^2\right) \, dt \\
        &= \frac{1}{2} \left(\|x(0) - x^l\|^2 - \|x(S) - x^l\|^2\right) \\
        &\leq \frac{1}{2} \|x(0) - x^l\|^2.
    \end{align*}
  Taking the limit as \( S \to \infty \), we obtain
\[
    \int_0^{+\infty}  \|x(t) - y(t)\|^2 \, dt < \infty.
\]

From the nonexpansiveness of \(P_K\), it follows that the trajectory \(y(t)\) is locally absolutely continuous and
\[
\|\dot{y}(t)\| \leq (1 + \lambda L)\|\dot{x}(t)\|, \quad \text{a.e. } t \geq 0.
\]
In addition,
\[
\|\dot{x}(t)\| = \|y(t) - x(t) + \lambda(F(x(t)) - F(y(t)))\|
\leq (1 + \lambda L)\|x(t) - y(t)\|.
\]
Hence,
\[
\|\dot{y}(t)\| \leq (1 + \lambda L)^2 \|x(t) - y(t)\|.
\]

Now consider the time derivative of the squared norm:
\[
\frac{d}{dt} \|x(t) - y(t)\|^2 = 2 \langle x(t) - y(t), \dot{x}(t) - \dot{y}(t) \rangle
\leq 2\|x(t) - y(t)\| \cdot \left(\|\dot{x}(t)\| + \|\dot{y}(t)\|\right).
\]
Thus,
\[
\frac{d}{dt} \|x(t) - y(t)\|^2 \leq 2\left(\|\dot{x}(t)\| + \|\dot{y}(t)\|\right)\|x(t) - y(t)\|
\leq C \|x(t) - y(t)\|^2,
\]
where \(C = 2\left((1 + \lambda L) + (1 + \lambda L)^2\right)\).

We have proved it before
\[
\int_0^{+\infty} \|x(t) - y(t)\|^2 \, dt < +\infty.
\]
Then, by Lemma \ref{lemma:asymptotic_decay}, we obtain
\[
\lim_{t \to +\infty} \|x(t) - y(t)\| = 0.
\]
\end{proof}
The following proposition can be viewed as a continuous analogue of Lemma 3.3 in our earlier work \cite{a25}.
\begin{proposition}\label{p3.5}
Suppose that assumptions \textbf{C1}-\textbf{C4} hold and $\lambda\in(0,\frac{1}{L})$. Then, one of the following alternatives must be true: either \( x^l \in \Omega_{D} \), or \( F(x^l) = 0 \), where \( x^l \) denotes a weak sequential cluster point of the trajectory \( x(t) \) as \( t \to +\infty \).
\end{proposition}

\begin{proof}
We have previously shown that the mapping \( t \mapsto \|x(t) - x^l\|^2 \) is nonincreasing, which implies the boundedness of the trajectory \( x(t) \) in the Hilbert space \( \mathcal{H} \). Hence, there exists a constant \( M > 0 \) such that \( \|x(t)\| \leq M \) for all \( t \geq 0 \).
\vskip 2mm 
Let \( x^l \in \mathcal{H} \) be a weak sequential cluster point of \( x(t) \), that is, there exists a sequence \( (t_n)_{n \geq 0} \subset [0, +\infty) \) with \( t_n \to +\infty \) and \( x(t_n) \rightharpoonup x^l \) weakly in \( \mathcal{H} \). From the established asymptotic behavior, we have
\[
\lim_{t \to +\infty} \|x(t) - y(t)\| = 0,
\]
implying that \( \|x(t_n) - y(t_n)\| \to 0 \) as \( n \to \infty \). Consequently, \( y(t_n) \rightharpoonup x^l \) weakly.
\vskip 2mm 
For simplicity, define \( x_n := x(t_n) \) and \( y_n := y(t_n) \) for all \( n \geq 0 \). We now consider two distinct cases.
\vskip 2mm
\textbf{Case 1:} Suppose \( \limsup_{n \to \infty} \|F(y_n)\| = 0 \). Then
\[
\lim_{n \to \infty} \|F(y_n)\| = \liminf_{n \to \infty} \|F(y_n)\| = 0.
\]
By the continuity assumption \textbf{C3}, we obtain
\[
0 \leq \|F(x^l)\| \leq \liminf_{n \to \infty} \|F(y_n)\| = 0,
\]
which implies \( F(x^l) = 0 \).
\vskip 2mm
\textbf{Case 2:} Suppose instead that \( \limsup_{n \to \infty} \|F(y_n)\| > 0 \). Without loss of generality, assume there exists a subsequence \( \{y_{n_k}\} \) such that
\[
\lim_{k \to \infty} \|F(y_{n_k})\| = M > 0.
\]
Then, there exists \( K \in \mathbb{N} \) such that \( \|F(y_n)\| > \frac{M}{2} \) for all \( n \geq K \). Since \( y_n = P_K(x_n - \lambda F(x_n)) \), by the characterization of projections, we have
\[
\langle y_n - x_n + \lambda F(x_n), z - y_n \rangle \geq 0, \quad \forall z \in K.
\]
This is equivalent to
\[
\langle x_n - y_n, z - y_n \rangle \leq \lambda \langle F(x_n), z - y_n \rangle, \quad \forall z \in K.
\]
Rewriting, we get
\[
\frac{1}{\lambda} \langle x_n - y_n, z - y_n \rangle - \langle F(x_n) - F(y_n), z - y_n \rangle \leq \langle F(y_n), z - y_n \rangle, \quad \forall z \in K.
\]

Fix any \( z \in K \). Taking the limit as \( n \to \infty \) and using the facts that \( \|x_n - y_n\| \to 0 \),  and \( \{y_n\} \) is bounded, we deduce:
\begin{equation} \label{ineq1}
    0 \leq \liminf_{n \to \infty} \langle F(y_n), z - y_n \rangle \leq \limsup_{n \to \infty} \langle F(y_n), z - y_n \rangle < +\infty.
\end{equation}

If \( \limsup_{n \to \infty} \langle F(y_n), z - y_n \rangle > 0 \), then there exists a subsequence \( \{y_{n_k}\} \) such that
\[
\lim_{k \to \infty} \langle F(y_{n_k}), z - y_{n_k} \rangle > 0.
\]
Thus, for sufficiently large \( k \), we have \( \langle F(y_{n_k}), z - y_{n_k} \rangle > 0 \). Invoking the quasimonotonicity of \( F \), it follows that
\[
\langle F(z), z - y_{n_k} \rangle \geq 0, \quad \forall k \geq k_0.
\]
Passing to the limit yields \( \langle F(z), z - x^l \rangle \geq 0 \), i.e., \( x^l \in \Omega_D \).
\vskip 2mm
If instead \( \limsup_{n \to \infty} \langle F(y_n), z - y_n \rangle = 0 \), then from \eqref{ineq1}, we deduce
\[
\lim_{n \to \infty} \langle F(y_n), z - y_n \rangle = 0.
\]
Define
\[
\varepsilon_n := |\langle F(y_n), z - y_n \rangle| + \frac{1}{n+1}.
\]
Clearly, \( \varepsilon_n > 0 \) and \( \varepsilon_n \to 0 \). Hence,
\begin{equation} \label{ineq2}
    \langle F(y_n), z - y_n \rangle + \varepsilon_n > 0.
\end{equation}
Let \( z_n := \frac{F(y_n)}{\|F(y_n)\|^2} \) for \( n \geq K \), so that \( \langle F(y_n), z_n \rangle = 1 \). From \eqref{ineq2}, it follows that
\[
\langle F(y_n), z + \varepsilon_n z_n - y_n \rangle > 0.
\]
By the quasimonotonicity of \( F \), we obtain
\[
\langle F(z + \varepsilon_n z_n), z + \varepsilon_n z_n - y_n \rangle \geq 0.
\]
Then
\begin{align*}
\langle F(z), z + \varepsilon_n z_n - y_n \rangle 
&= \langle F(z) - F(z + \varepsilon_n z_n), z + \varepsilon_n z_n - y_n \rangle \\
&\quad + \langle F(z + \varepsilon_n z_n), z + \varepsilon_n z_n - y_n \rangle \\
&\geq -\|F(z) - F(z + \varepsilon_n z_n)\| \cdot \|z + \varepsilon_n z_n - y_n\| \\
&\geq -\varepsilon_n L \|z_n\| \cdot \|z + \varepsilon_n z_n - y_n\| \\
&= -\varepsilon_n \frac{L}{\|F(y_n)\|} \cdot \|z + \varepsilon_n z_n - y_n\| \\
&\geq -\varepsilon_n \frac{2L}{M} \cdot \|z + \varepsilon_n z_n - y_n\|.
\end{align*}
Since \( \varepsilon_n \to 0 \) and the norm \( \|z + \varepsilon_n z_n - y_n\| \) remains bounded, we conclude
\[
\langle F(z), z - x^l \rangle \geq 0, \quad \forall z \in K.
\]
Thus, \( x^l \in \Omega_D \), as claimed.
\end{proof}

\begin{lemma}\label{l3.7}
Assume that condition \textbf{C4} hold, that is, for every \( p \in \Omega_D \), let $A_p := \{p\} \cup A \subset \mathcal{H}$ be a finite subset of a Hilbert space $\mathcal{H}$, and suppose there exists a constant $\delta_p > 0$ such that
\[
\|y_1 - y_2\| \geq 2\delta_p, \quad \forall\, y_1 \neq y_2 \in A_p.
\]
Then, there exists a finite set of unit vectors $\{r_1, \dots, r_m\} \subset \mathcal{H}$ and the same constant $\delta_p > 0$ such that the tube neighborhoods
\[
B(y, \delta_p) := \left\{ x \in H : |\langle r_i, x - y \rangle| < \delta_p, \quad i = 1, \dots, m \right\}
\]
satisfy the property:
\[
B(y_1, \delta_p) \cap B(y_2, \delta_p) = \emptyset, \quad \forall\, y_1 \neq y_2 \in A_p.
\]
\end{lemma}

\begin{proof}
Since $A_p$ is a finite set, the number of unordered distinct point pairs $(y_1, y_2)$ with $y_1 \neq y_2 \in A_p$ is finite. For each such pair, define the unit direction vector
\[
r_{(y_1,y_2)} := \frac{y_2 - y_1}{\|y_2 - y_1\|}.
\]
Collect all such distinct directions (after removing repetitions due to symmetry), and denote the resulting finite set as $\{r_1, \dots, r_m\}$.
We claim that for the above choice of directions and the constant $\delta_u$ from the assumption, the sets $B(y, \delta_p)$ are disjoint for distinct $y$.

Assume, for contradiction, that there exists $x \in \mathcal{H}$ such that $x \in B(y_1, \delta_p) \cap B(y_2, \delta_p)$ for some $y_1 \neq y_2 \in A_p$. Then, in particular, for the unit vector
\[
r := \frac{y_2 - y_1}{\|y_2 - y_1\|},
\]
we have
\[
|\langle r, x - y_1 \rangle| < \delta_p, \quad |\langle r, x - y_2 \rangle| < \delta_p.
\]

However, note that
\[
\langle r, x - y_2 \rangle = \langle r, x - y_1 \rangle + \langle r, y_1 - y_2 \rangle = \langle r, x - y_1 \rangle - \|y_2 - y_1\|.
\]
Taking absolute value yields
\[
|\langle r, x - y_2 \rangle| = \left| \langle r, x - y_1 \rangle - \|y_2 - y_1\| \right| \geq \|y_2 - y_1\| - |\langle r, x - y_1 \rangle|.
\]
Since $|\langle r, x - y_1 \rangle| < \delta_p$ and by assumption $\|y_2 - y_1\| \geq 2\delta_p$, we obtain
\[
|\langle r, x - y_2 \rangle| > 2\delta_p - \delta_p = \delta_p,
\]
which contradicts $x \in B(y_2, \delta_p)$. Therefore, such $x$ cannot exist, and we conclude
\[
B(y_1, \delta_p) \cap B(y_2, \delta_p) = \emptyset, \quad \forall\, y_1 \neq y_2 \in A_p.
\]
\end{proof}
\begin{lemma}
Assume that conditions \textbf{C1}-\textbf{C4} hold and $\lambda\in(0,\frac{1}{L})$. Then the trajectory \(\{x(t)\}\) admits at most one weak cluster point in \(\Omega_D\) as \(t \to \infty\). Moreover, there exists a constant \(N_3 \geq 0\) such that for all \(t \geq N_3\), the trajectory $x(t)$ lies in the set \(\Omega := \bigcup_{y \in A_p} B\left(y, \frac{\delta_p}{2}\right)\), and there exists a constant \(T \geq N_3\) such that $y(t)\in\Omega$ when \(t \geq T\).
\end{lemma}

\begin{proof}
Suppose, by contradiction, that the trajectory $\{x(t)\}_{t \geq 0}$ admits at least two distinct weak cluster points in $\Omega_D$, say $\tilde{x} \neq \bar{x}$. Let $\{x(t_i)\}$ be a subsequence of $\{x(t)\}$ that converges weakly to $\tilde{x}$ as $i\to\infty$. According to Proposition \ref{p3.4}, the limit of $\|x(t)-x^l\|$ exists for all $x^l\in\Omega_{D}$. Hence, we obtain the following:
\begin{align*}
    \lim_{t\to\infty}\|x(t)-\tilde{x}\|& = \lim_{i\to\infty}\|x(t_i)-\tilde{x}\| = \liminf_{i\to\infty}\|x(t_i)-\tilde{x}\|\\
    &<\liminf_{i\to\infty}\|x(t_i)-\bar{x}\| =  \lim_{t\to\infty}\|x(t)-\bar{x}\|\\
    & = \lim_{j\to\infty}\|x(t_j)-\bar{x}\| = \liminf_{j\to\infty}\|x(t_j)-\bar{x}\|\\
    &< \liminf_{j\to\infty}\|x(t_j)-\tilde{x}\| = \lim_{j\to\infty}\|x(t_j)-\tilde{x}\| = \lim_{t\to\infty}\|x(t)-\tilde{x}\|,
\end{align*}
which leads to contradiction. 
\vskip 2mm
Next, we show that the trajectory eventually enters the set \(\Omega\). Assume for contradiction that there exists a sequence \(\{t_k\}\) with \(t_k \to \infty\) such that \(x(t_k) \notin \Omega\) for all \(k\). Since \(\{x(t_k)\}\) is bounded, there exists a weakly convergent subsequence (still denoted \(\{x(t_k)\}\)) such that \(x(t_k) \rightharpoonup x^*\).
Note that \(\Omega\) is an open set. Since \(x(t_k) \notin \Omega\), it follows from the weak limit that \(x^* \notin \Omega\). By definition of \(\Omega\), this implies \(x^* \notin A_p\). However, the definition of $A_p$ ensures that any weak cluster point must lie in \(A_p\), yielding a contradiction. Hence, we conclude that \(x(t) \in \Omega\) for all sufficiently large \(t\), i.e., there exists \(N_3 \geq 0\) such that \(x(t) \in \Omega\) for all \(t \geq N_3\).
\vskip 2mm 
 Indeed, since \(\Omega\) is an open set and \(x(t) \in \Omega\) for \(t \geq N_3\), for each such \(t\), there exists \(\varepsilon_t > 0\) such that the open ball \(B(x(t), \varepsilon_t) \subset \Omega\). By the assumption \(\|x(t) - y(t)\| \to 0\), there exists a constant $T\geq N_3$, for sufficiently large \(t\geq T\geq N_3\), we have \(\|x(t) - y(t)\| < \varepsilon_t\), which ensures \(y(t) \in \Omega\). Hence, the trajectory \(\{y(t)\}\) also eventually lies entirely in \(\Omega\).
\end{proof}

\begin{proposition}
       Assume that conditions \textbf{C1}-\textbf{C4} hold and $\lambda\in(0,\frac{1}{L})$. Then, the trajectories \(x(t)\) and \(y(t)\) generated by \eqref{4} converge weakly to a solution of \(\Omega\) as \(t \to \infty\).
\end{proposition}
\begin{proof}
From Proposition~\ref{p3.4}, we know that \(\lim_{t \to \infty} \|x(t) - y(t)\| = 0\). Thus, there exists \(N_4 > N_3\) such that for all \(t \geq N_4\),
\begin{equation*}\label{3.7}
   \|x(t) - y(t)\| < \frac{\delta_p}{4}. \tag{3.7} 
\end{equation*}
Now, assume by contradiction that the sequence \(\{x(t)\}\) has more than one weak cluster point. Let \(y_1\) and \(y_2\) be two distinct weak cluster points of \(\{x(t)\}\) as \(t \to \infty\). According to Lemma~\ref{l3.7}, there exists \(N_5\geq T \geq N_4 > N_3\) such that \(x(t)_{N_5} \in \Omega(y_1, \frac{\delta_p}{2})\) and \(y(t)_{N_5} \in \Omega(y_2, \frac{\delta_p}{2})\). This implies that for \(i = 1, 2\),
\begin{equation*}\label{3.8}
    -\frac{\delta_p}{2} < \langle r_i, x(t)_{N_5} - y_1 \rangle < \frac{\delta_p}{2}. \tag{3.8}
\end{equation*}
Since \(y(t)_{N_5} \in \Omega(y_2, \frac{\delta_p}{2})\), it follows that \(y(t)_{N_5} \in \Omega(y_2, \delta_p)\). By the property of \(\Omega\), we have \(\Omega(y_1, \delta_p) \cap \Omega(y_2, \delta_p) = \emptyset\), which implies that \(y(t)_{N_5} \notin \Omega(y_1, \delta_p)\). Therefore, there exists \(r_j\) (for some \(j \in \{1, \dots, m\}\)) such that
\begin{equation*}\label{3.9}
    \langle r_j, y(t)_{N_5} - y_1 \rangle \geq \delta_p \quad \text{or} \quad \langle r_j, y(t)_{N_5} - y_1 \rangle \leq -\delta_p. \tag{3.9}
\end{equation*}
Substituting \(i = j\) into \eqref{3.8} and combining with \eqref{3.9}, we obtain
\[
\frac{\delta_p}{2} \leq |\langle r_j, y(t)_{N_5} - x(t)_{N_5} \rangle| \leq \|r_j\| \|y(t)_{N_5} - x(t)_{N_5}\| = \|y(t)_{N_5} - x(t)_{N_5}\|.
\]

From \eqref{3.7}, we know that
\[
\|y(t)_{N_5} - x(t)_{N_5}\| < \frac{\delta_p}{4}.
\]

This leads to a contradiction, as it follows that
\[
\frac{\delta_p}{2} \leq \|y(t)_{N_5} - x(t)_{N_5}\| < \frac{\delta_p}{4}.
\]

Therefore, the sequence \(\{x(t)\}\) must have exactly one weak cluster point in \(\Omega\). 
 Since \( \|x(t) - y(t)\| \to 0 \) as \( t \to \infty \), and \( \{x(t)\} \) has a unique weak cluster point \( x^* \), it follows that \( y(t) \) must also converge weakly to \( x^* \) as \( t \to \infty \).
This completes the proof.
\end{proof}

\begin{proposition}\label{p3.9}
    Suppose that assumptions \textbf{C1}–\textbf{C4} hold, $\lambda\in(0,\frac{1}{L})$, as well as the following condition:
\vskip 2mm 
\textbf{C5:} The sequence $\{y(t)\}\subset K$. for $\epsilon\geq0$, there is 
\begin{equation*}
   \liminf_{t\to\infty}\frac{\vert \langle F(y(t)),y(t)-u\rangle\vert}{\|y(t)-u\|^{2+\epsilon}}>0, \forall u\in K.  
\end{equation*}
\vskip 2mm 
Then, the trajectory $\{x(t)\}$ generated by \eqref{4} converges strongly to a point in $\Omega$.
\end{proposition}
\begin{proof}
    By assumption \textbf{C5}, there exist constants $c>0$ and $N'\in\mathbb{N}$ such that for all $t > N'$, we have
    \begin{equation*}\label{3.10}
        \vert \langle F(y(t)),y(t)-u\rangle\vert\geq c\|y(t)-u\|^{2+\epsilon}\geq 0. \tag{3.10}
    \end{equation*}
For the purpose of discussion, we will divide the proof into two cases as follows for
discussion.
\vskip 2mm 
\textbf{Case I:} If $x^*\in \Omega_{D} $, then from $y(t)\in K$, we can obtain 
\begin{equation*}
    \langle F(y(t)), y(t)-x^*\rangle\geq 0, 
\end{equation*}
hence, combining with \eqref{3.10}, we have 
\begin{equation*}
     \langle F(y(t)),y(t)-u\rangle\geq c\|y(t)-u\|^{2+\epsilon}\geq 0.
\end{equation*}
According to the Lipschitz continuous property of $F$, we get 
\begin{align*}
     \langle F(x(t)),x^*-y(t)\rangle& =  \langle F(x(t))-F(y(t)),x^*-y(t)\rangle- \langle F(y(t)),y(t)-x^*\rangle\\
     &\leq \| F(x(t))-F(y(t))\|\|y(t)-x^*\|-c\|y(t)-x^*\|^{2+\epsilon}\\
     &\leq L\|x(t)-y(t)\|\|y(t)-x^*\|-c\|y(t)-x^*\|^{2+\epsilon}, 
\end{align*}
which, in combination with $ \langle x(t) - \lambda F(x(t)) - y(t), y(t) - x^* \rangle \geq 0, \quad \forall t \in [0, +\infty)$, gives 
\begin{align*}
    \langle x^*-y(t),x(t)-y(t)\rangle&\leq \lambda\langle F(x(t)), x^*-y(t)\rangle\\
    &\leq \lambda L\|x(t)-y(t)\|\|y(t)-x^*\|-c\lambda\|y(t)-x^*\|^{2+\epsilon}. 
\end{align*}
which implies that 
\begin{align*}
    c\lambda\|y(t)-x^*\|^{2+\epsilon}&\leq \lambda L\|x(t)-y(t)\|\|y(t)-x^*\|- \langle x^*-y(t),x(t)-y(t)\rangle\\
    &\leq\lambda L\|x(t)-y(t)\|\|y(t)-x^*\|+\|x^*-y(t)\|\|x(t)-y(t)\|\\
     & = \left(\lambda L+1\right)\|x^*-y(t)\|\|x(t)-y(t)\|. 
\end{align*}
Due to $\|x(t)-y(t)\|\to 0$ as $t\to\infty$, we derive that 
\begin{align*}
    \|x(t)-x^*\|&\leq \|x(t)-y(t)\|+\|y(t)-x^*\|\\
    &\leq  \|x(t)-y(t)\|+\left[\frac{1+\lambda L}{c\lambda}\right]^{\frac{1}{1+\epsilon}}\cdot\|x(t)-y(t)\|^{\frac{1}{1+\epsilon}}\to 0, \quad \text{as} \quad t\to\infty.
\end{align*}
Thus, $x(t)\to x^*$ strongly.
\vskip 2mm 
\textbf{Case II:} If $x^*\notin\Omega_{D}$, i.e., $x^*\in A$, Next, we prove that $\|F(y(t))\|\to 0$ as $t\to\infty$. 
\vskip 2mm 
Since $x(t)\rightharpoonup x^*$ and $\|x(t)-y(t)\|\to 0$ as $t\to\infty$, thus  $y(t)\rightharpoonup x^*\in K$. Based on the Lipschitz continuity of $F$ and the sequence $y(t)$ is bounded, then $0\leq \limsup_{t\to\infty}\|F(y(t))\|<+\infty$. We adopt the method of proof by contradiction. Assume that $\limsup_{t\to\infty}\|F(y(t))\|\neq 0$, i.e., $\limsup_{t\to\infty}\|F(y(t))\|> 0$. Similar to the proof of Proposition \ref{p3.5}, we can prove that $x^*\in\Omega_{D}$, which contradicts $x^*\notin \Omega_{D}$, hence, $\|F(y(t))\|\to 0$ as $t\to\infty$. Therefore, by \eqref{3.10}, we obtain 
\begin{equation*}
    c\|y(t)-x^*\|^{2+\epsilon}\leq  \vert \langle F(y(t)),y(t)-x^*\rangle\vert\leq \|F(y(t))\|\|y(t)-x^*\|. 
\end{equation*}
Assume that the differential system does not stop within a finite number of steps, i.e., $\|y(t)-x^*\|\neq 0$, then 
\begin{equation*}
     c\|y(t)-x^*\|^{1+\epsilon}\leq \|F(y(t))\|, 
\end{equation*}
which means that $\|y(t)-x^*\|\to 0$ as $t\to\infty$. Hence, 
\begin{align*}
    \|x(t)-x^*\|\leq \|x(t)-y(t)\|+\|y(t)-x^*\|\to 0, \quad \text{as} \quad t\to\infty. 
\end{align*}
In conclusion, we finally obtain that the trajectory $x(t)$ converges strongly to $x^*\in \Omega$.     
\end{proof}

\begin{remark}
In above analysis, we introduced condition \textbf{C5} to replace the commonly used assumption of strong pseudomonotonicity for the operator $F$. This condition only requires an asymptotic nondegeneracy along the trajectory $\{y(t)\}$, namely that the normalized gap $\frac{|\langle F(y(t)), y(t) - u \rangle|}{\|y(t) - u\|^{2 + \epsilon}}$ remains bounded away from zero as $t \to \infty$. Importantly, \textbf{C5} is a significantly weaker requirement compared to strong pseudomonotonicity, which is a global property imposing a uniform quadratic growth lower bound for the residual across the entire feasible set $K$. Our result shows that under this milder assumption, the continuous dynamical system \eqref{4} still achieves strong convergence of its trajectory to a solution of the variational inequality problem.
\end{remark}
\vskip 2mm
We have established the weak convergence of the trajectories generated by the proposed dynamical system under mild assumptions, and further proved strong convergence under additional conditions. To complement these results, we now investigate the \emph{ergodic convergence} of the system.
Although strong convergence already guarantees the pointwise limit of the trajectory \(x(t)\), analyzing the ergodic behavior provides additional insights into the average long-term dynamics of the system. Such results can be useful, for instance, in evaluating the global stability of the method, or in justifying the effectiveness of time-averaged outputs in practical implementations.
\begin{theorem}
    Suppose that assumptions \textbf{C1}–\textbf{C5} are satisfied and $\lambda\in(0,\frac{1}{L})$. Then, the time-average \( \bar{x}(T) \) defined by
    \[
    \bar{x}(T) := \frac{1}{T} \int_0^T x(t)\,dt
    \]
    converges strongly to a solution of $\Omega$ as \( T \to \infty \).
\end{theorem}
\begin{proof}
From above analysis, we know $x(t)\to x^*$ and $\|x(t)-y(t)\|\to 0$ as $t\to\infty$. Hence, fix arbitrary $\varepsilon>0$. Then by the strong convergence of $x(t)$, there exists $T_1>0$ such that
\begin{equation*}
        \|x(t)-x^*\|<\frac{\varepsilon}{2}, \quad \forall t\ge T_1. 
\end{equation*}
Also, from $\int_0^{\infty}\|x(t)-y(t)\|^2dt<\infty$, there exists $T_2>0$ such that 
\begin{align*}
    \int_{T_2}^{\infty}\|x(t)-y(t)\|^2\,dt<\frac{\varepsilon^2}{8}.
\end{align*}
Let $T_0=\max\{T_1,T_2\}$. We split the time-average into two parts:  
\begin{equation*}
    \bar{x}(T)-x^* = \frac{1}{T} \int_0^{T_0} [x(t)-x^*]\,dt + \frac{1}
{T} \int_{T_0}^{T} [x(t)-x^*]\,dt. 
\end{equation*}
Taking norms and using the triangle inequality, we have 
\begin{equation*}
       \| \bar{x}(T)-x^*\|\leq \frac{1}{T} \int_0^{T_0} \|x(t)-x^*\|\,dt + \frac{1}
{T} \int_{T_0}^{T} \|x(t)-x^*\|\,dt =: I_{1}(T)+I_{2}(T). 
\end{equation*}
For the first term, since $x(t)$ is bounded, say $\|x(t)-x^*\|\le M$, then 
\begin{equation*}
    I_{1}(T) = \frac{1}{T} \int_0^{T_0} \|x(t)-x^*\|\,dt\leq \frac{1}{T} \int_0^{T_0} M\,dt=\frac{T_{0}}{T}M\to 0, \quad \text{as}\quad T\to\infty.
\end{equation*}
For the second term, split further using triangle inequality, which yields 
\begin{equation*}
     I_2(T) \le \frac{1}{T} \int_{T_0}^T \|x(t)-y(t)\|\,dt + \frac{1}{T} \int_{T_0}^T \|y(t)-x^*\|\,dt \ =: J_1(T) + J_2(T). 
\end{equation*}
Let \( f(t) = \|x(t) - y(t)\| \) and \( g(t) \equiv 1 \). By Cauchy-Schwarz inequality, we obtain 
\[
\int_{T_0}^T \|x(t) - y(t)\|\,dt = \int_{T_0}^T f(t)\cdot g(t)\,dt \le 
\left( \int_{T_0}^T \|x(t) - y(t)\|^2\,dt \right)^{1/2}
\left( \int_{T_0}^T 1^2\,dt \right)^{1/2}.
\]
Since \( \int_{T_0}^T 1^2\,dt = T - T_0 \), we get 
\[
\int_{T_0}^T \|x(t) - y(t)\|\,dt \le \sqrt{T - T_0} \left( \int_{T_0}^T \|x(t) - y(t)\|^2\,dt \right)^{1/2}, 
\]
which means that 
\begin{align*}
    J_{1}(T) &=  \frac{1}{T}\int_{T_0}^T \|x(t) - y(t)\|\,dt\leq  \frac{1}{T} \sqrt{T - T_0} \left( \int_{T_0}^T \|x(t) - y(t)\|^2\,dt \right)^{1/2}\\
    & \leq \frac{1}{T} \sqrt{T} \left( \int_{T_0}^T \|x(t) - y(t)\|^2\,dt \right)^{1/2}\leq \frac{1}{\sqrt{T}}\cdot \frac{\varepsilon}{2\sqrt{2}} = \frac{\varepsilon}{2\sqrt{2T}}. 
\end{align*}
Hence,  For large $T$, we have 
\begin{align*}
    J_{1}(T)<\frac{\varepsilon}{4}.
\end{align*}
Since for $t\ge T_0\ge T_1$, we have $\|x(t)-x^*\|<\frac{\varepsilon}{2}$, we get for all $t\ge T_0$: 
\begin{equation*}
     \|y(t)-x^*\| = \|y(t)-x(t)\|+\|x(t)-x^*\| < \frac{\varepsilon}{2} + \|x(t)-y(t)\|.
\end{equation*}
Thus 
\begin{align*}
    J_2(T)& = \frac{1}{T} \int_{T_0}^T \|y(t)-x^*\|\,dt 
    \le \frac{1}{T} \int_{T_0}^T \Bigl(\frac{\varepsilon}{2} + \|x(t)-y(t)\|\Bigr) dt\\
    &= \frac{T- T_0}{T}\frac{\varepsilon}{2} + J_1(T)<\frac{\varepsilon}{2}+\frac{\varepsilon}{4} = \frac{3\varepsilon}{4}.
\end{align*}
Combining these estimates, for sufficiently large $T$, 
\begin{equation*}
     \|\bar{x}(T)-x^*\| \le I_1(T) + J_1(T) + J_2(T) < 0 +
\frac{\varepsilon}{4} + \frac{3\varepsilon}{4} = \varepsilon. 
\end{equation*}
Thus $\bar{x}(T)\to x^*$ as $T\to\infty$, which completes the proof of ergodic convergence.
\end{proof}

\section{The forward-backward-forward algorithm based on the golden ratio and relaxation parameters}\label{sec:algorithm}
\quad\quad By applying an explicit time discretization to the dynamical system \eqref{4} with stepsize $\gamma_k > 0$ and an initial point $x_0 \in \mathcal{H}$, we obtain the following recursive formula for each iteration $k \geq 0$:
\[
x_k + \frac{x_{k+1} - x_k}{\gamma_k} = P_K(x_k - \lambda F(x_k)) + \lambda \left[ F(x_k) - F\left(P_K(x_k - \lambda F(x_k))\right) \right].
\]
To simplify notation and highlight the projection step, we define
\[
y_k := P_K(x_k - \lambda F(x_k)).
\]
Substituting this into the previous equation yields the equivalent iterative scheme:
\[
\begin{cases}
y_k = P_K(x_k - \lambda F(x_k)), \\
x_{k+1} = \gamma_k \left( y_k + \lambda (F(x_k) - F(y_k)) \right) + (1 - \gamma_k) x_k.
\end{cases}
\]
This iterative process corresponds precisely to the \emph{relaxed forward-backward-forward algorithm}, where $(\gamma_k)_{k \geq 0}$ acts as a sequence of relaxation parameters. Notably, when the relaxation parameter is set to $\gamma_k = 1$ for all $k$, the method reduces to the classical forward-backward-forward (FBF) algorithm originally proposed in \cite{a13}. 
\vskip 2mm 
In this section, we explore the iterative approximation of solutions to quasimonotone variational inequalities within reflexive Banach spaces by employing the Bregman distance function and Bregman projection. Unlike most existing works that assume Lipschitz continuity of the operator \(F\), our method only requires uniform continuity. Prior algorithms that drop the Lipschitz assumption typically rely on line search techniques to determine stepsize. In contrast, we propose an adaptive stepsize strategy that avoids line search, offering lower computational cost and broader applicability. To the best of our knowledge, this is the first Bregman-type relaxed FBF method under uniform continuity using adaptive stepsizes.
\vskip 2mm 
To improve the performance of the relaxed forward-backward-forward (FBF) algorithm, we incorporate an additional extrapolation step inspired by the golden ratio technique before each main iteration. Specifically, prior to computing $y_k$ and $x_{k+1}$, we introduce an extrapolated point $w_k$, defined as a convex combination of the current and previous iterates, guided by a golden ratio-based parameter. This extrapolation strategy has been shown in various studies to accelerate the convergence of iterative schemes.
\vskip 2mm 
The following conditions are required. 
\vskip 2mm 
\textbf{(D1)}: The solution set \( \Omega_{D} \) is nonempty; that is, \( \Omega_{D} \neq \emptyset \).

\vskip 2mm
\textbf{(D2)}: The operator \( F:E \to E^* \) is quasimonotone and \textbf{uniformly continuous} on $E$.
\vskip 2mm
\vskip 2mm
\textbf{(D3)}: The function $\Phi:E\to \mathbb{R}$ is strongly convex with modulus $\rho>\rho_{0}>0$, where $\rho_{0}$ is a given constant depending on the algorithm parameters. It is also Legendre, proper, and lower semicontinuous; uniformly Fréchet differentiable; and bounded on every bounded subset of $K$.
\vskip 2mm 
\textbf{(D4)}: The mapping \( F \) satisfies a sequential lower semicontinuity condition: whenever a sequence \( \{x_k\} \subset E \) converges weakly to \( x \), i.e., \( x_k \rightharpoonup x \), it holds that
\[
\|F(x)\| \leq \liminf_{k \to \infty} \|F(x_k)\|.
\]
\\
\noindent\rule[0.25\baselineskip]{\textwidth}{0.5pt}\\
\textbf{Algorithm 1}\\
\noindent\rule[0.25\baselineskip]{\textwidth}{0.5pt}

\vskip 1mm
\textbf{Initialization:} 
Choose parameters $\gamma_k = \gamma \in (0,1)$, $\mu \in (0,1)$, and $\psi \in (1, +\infty)$. 
Let $x_0, x_1 \in E$ be initial points, and let $\{\eta_k\}_{k\in\mathbb{N}} \subset [0, \infty)$ satisfy $\sum_{k=1}^{\infty} \eta_k < \infty$.

\vskip 2mm
\textbf{Iterative Steps:} For each iteration $k \geq 1$, compute $x_{k+1}$ as follows:

\vskip 2mm
\textbf{Step 1.} Compute extrapolation, projection, correction, and update:
\begin{equation*}
\begin{cases}
w_k = \nabla \Phi^{*} \left( \dfrac{(\psi - 1)\nabla \Phi(x_k) + \nabla \Phi(w_{k-1})}{\psi} \right), \\[1.5ex]
y_k = \operatorname{proj}_K^{\Phi} \left( \nabla \Phi^{*} \left( \nabla \Phi(w_k) - \lambda_k F(w_k) \right) \right), \\[1.5ex]
x_{k+1} = \nabla \Phi^{*} \left( (1 - \gamma) \nabla \Phi(x_k) + \gamma \left[ \nabla \Phi(y_k) - \lambda_k (F(y_k) - F(w_k)) \right] \right).
\end{cases}
\end{equation*}

\vskip 2mm
\textbf{Step 2.} Update the stepsize $\lambda_{k+1}$ by
\begin{equation}\label{e18}
\lambda_{k+1} =
\begin{cases}
\min\left\{ \dfrac{\mu \|w_k - y_k\|}{\|F(w_k) - F(y_k)\|},\; \lambda_k + \eta_k \right\}, & \text{if } F(w_k) \neq F(y_k), \\[1.5ex]
\lambda_k  + \eta_k, & \text{otherwise}.
\end{cases}
\end{equation}

\vskip 2mm
\textbf{Step 3.} Set $k := k + 1$ and return to \textbf{Step 1}.
\vskip 2mm

\noindent\rule[0.25\baselineskip]{\textwidth}{0.5pt}
\vskip 2mm 
\begin{lemma}\label{l4.1}
Let the sequence $\{\lambda_k\}$ be generated by Algorithm~1, and suppose that Assumption \textbf{(D2)} holds. Then:
\begin{itemize}
    \item[(i)] $\{\lambda_k\}$ is bounded above by $\lambda_1 + \Xi$, where $\Xi := \sum_{k=1}^\infty \eta_k$;
    \item[(ii)] $\{\lambda_k\}$ converges to some $\lambda \geq 0$, i.e., $\lim_{k \to \infty} \lambda_k = \lambda$;
    \item[(iii)] The series $\sum_{k=0}^\infty (\lambda_{k+1} - \lambda_k)_+$ and $\sum_{k=0}^\infty (\lambda_{k+1} - \lambda_k)_-$ both converge, where
    \[
    (\lambda_{k+1} - \lambda_k)_+ := \max\{0, \lambda_{k+1} - \lambda_k\}, \quad
    (\lambda_{k+1} - \lambda_k)_- := \max\{0, \lambda_k - \lambda_{k+1}\}.
    \]
\end{itemize}
\end{lemma}

\begin{proof}
The proof of this lemma can refer to our earlier work \cite{a25}. The main difference is that we cannot obtain the lower bound value of the step size $\{\lambda_{k}\}$ when the operator $F$ is uniformly continuous.
\end{proof}
For the convergence analysis, we suppose that Algorithm 1 does not terminate after a finite number of iterations. 
\begin{lemma}\label{l4.2}
Assume that conditions \textbf{(D1)}-\textbf{(D3)} hold. Let $z_k = \nabla \Phi^* \left( \nabla \Phi(y_k) - \lambda_k \left(F(y_k) - F(w_k)\right) \right)$ and $\rho_{0}=1$. Then the sequence $\{x_k\}$ generated by Algorithm 1 satisfies the following:
\vskip 2mm
(1) The following limits hold:
\[
\lim_{k \to \infty} \|w_k - w_{k-1}\| = 
\lim_{k \to \infty} \|w_k - x_k\| = 
\lim_{k \to \infty} \|w_k - y_k\| = 
\lim_{k \to \infty} \|z_k - y_k\| = 0.
\]
(2) For any $p \in \Omega_D$, the Bregman distance $D_\Phi(p, x_k)$ converges; that is,
$\lim_{k \to \infty} D_\Phi(p, x_k) \text{ exists}$.
\end{lemma}
\begin{proof}
By the definition of the Bregman distance, we have
\begin{align*}
    D_{\Phi}(p, z_k) 
    &= D_{\Phi}\left(p, \nabla \Phi^*\left(\nabla \Phi(y_k) - \lambda_k (F(y_k) - F( w_k))\right)\right) \\
    &= \Phi(p) - \langle p - z_k, \nabla \Phi(y_k) - \lambda_k (F(y_k) - F(w_k)) \rangle - \Phi(z_k) \\
    &= \Phi(p) - \langle p - y_k, \nabla \Phi(y_k) \rangle - \Phi(y_k) + \langle p - y_k, \nabla \Phi(y_k) \rangle + \Phi(y_k) - \Phi(z_k) \\
    &\quad - \langle p - z_k, \nabla \Phi(y_k) \rangle - \lambda_k \langle p - z_k, F(w_k) - F( y_k) \rangle \\
    &= D_{\Phi}(p, y_k) - \Phi(z_k) + \Phi(y_k) - \langle p - z_k, \nabla \Phi(y_k) \rangle \\
    &\quad + \langle p - y_k, \nabla \Phi(y_k) \rangle - \lambda_k \langle p - z_k, F(w_k) - F( y_k) \rangle \\
    &= D_{\Phi}(p, y_k) - D_{\Phi}(z_k, y_k) + \lambda_k \langle z_k - p, F(w_k) - F(y_k) \rangle. \tag{4.1}\label{4.1}
\end{align*}

Next, applying the three-point identity \eqref{bregman_identity} for the Bregman distance, we derive
\begin{align*}
    D_{\Phi}(p, y_k) 
    &= D_{\Phi}(p, w_k) - D_{\Phi}(y_k, w_k) + \langle p - y_k, \nabla \Phi(w_k) - \nabla \Phi(y_k) \rangle. \tag{4.2}\label{4.2}
\end{align*}

Combining equations \eqref{4.1} and \eqref{4.2}, we obtain
\begin{align*}
    D_{\Phi}(p, z_k) 
    &= D_{\Phi}(p, w_k) - D_{\Phi}(y_k, w_k) - D_{\Phi}(z_k, y_k) 
    + \langle p - y_k, \nabla \Phi(w_k) - \nabla \Phi(y_k) \rangle \\
    &\quad + \lambda_k \langle z_k - y_k, F w_k - F y_k \rangle + \lambda_k \langle p - y_k, F y_k \rangle. \tag{4.3}\label{4.3}
\end{align*}

Since \( p \in \Omega_D \subset \Omega \subset K \) and by the definition of \( y_k \), together with the obtuse-angle property, we have the inequality
\begin{equation*}
    \langle p - y_k, \nabla \Phi(w_k) - \lambda_k B w_k - \nabla \Phi(y_k) \rangle \leq 0. \tag{4.4}\label{4.4}
\end{equation*}

Moreover, as both \( y_k \in K \) and \( p \in \Omega_{D} \subset K \), we also have
\begin{equation*}
    \lambda_k \langle F y_k, p - y_k \rangle \leq 0. \tag{4.5}\label{4.5}
\end{equation*}

Substituting inequalities \eqref{4.4} and \eqref{4.5} into equation \eqref{4.3}, we deduce
\begin{align*}
    D_{\Phi}(p, z_k) 
    &\leq D_{\Phi}(p, w_k) - D_{\Phi}(y_k, w_k) - D_{\Phi}(z_k, y_k) + \lambda_k \langle z_k - y_k, F(w_k) - F(y_k) \rangle. \tag{4.6}\label{4.6}
\end{align*}
Next, we will consider the following two situations. 
\vskip 2mm 
\textbf{Case I:} Assume that $\lambda_{k}\to \lambda>0$ as $k\to\infty$. According to \eqref{4.6} and $D_{\Phi}(x,y)\geq \frac{\rho}{2}\|x-y\|^2\geq\frac{1}{2}\|x-y\|^2$, we obtain 
\begin{align*}
    D_{\Phi}(p, z_k) 
    &\leq D_{\Phi}(p, w_k) - D_{\Phi}(y_k, w_k) - D_{\Phi}(z_k, y_k) + \lambda_k \langle z_k - y_k, F(w_k) - F(y_k) \rangle\\
    &\leq D_{\Phi}(p, w_k) - D_{\Phi}(y_k, w_k) - D_{\Phi}(z_k, y_k) + \lambda_k \| z_k - y_k\| \|F(w_k) - F(y_k)\|\\
    & \leq D_{\Phi}(p, w_k) - D_{\Phi}(y_k, w_k) - D_{\Phi}(z_k, y_k) + \frac{\mu\lambda_{k}}{\lambda_{k+1}} \| z_k - y_k\| \|w_k - y_k\|\\
    & \leq D_{\Phi}(p, w_k) - D_{\Phi}(y_k, w_k) - D_{\Phi}(z_k, y_k) + \frac{\mu\lambda_{k}}{2\lambda_{k+1}} \| z_k - y_k\|^2 +\frac{\mu\lambda_{k}}{2\lambda_{k+1}}\|w_k - y_k\|^2\\
    & \leq D_{\Phi}(p, w_k) -\left(1-\frac{\mu\lambda_{k}}{\rho\lambda_{k+1}}\right) D_{\Phi}(y_k, w_k) -\left(1-\frac{\mu\lambda_{k}}{\rho\lambda_{k+1}}\right) D_{\Phi}(z_k, y_k).\tag{4.7}\label{4.7}
\end{align*}
It follows from \eqref{2.3} and $\nabla \Phi(x_{k})=\frac{\psi}{\psi-1}\nabla {\Phi}(w_{k})-\frac{1}{\psi-1}\nabla{\Phi}(w_{k-1})$ that 
\begin{equation*}
    D_{\Phi}(p,x_{k}) = \frac{\psi}{\psi-1}\left[D_{\Phi}(p,w_{k})-D_{\Phi}(x_{k},w_{k})\right]-\frac{1}{\psi-1}\left[D_{\Phi}(p,w_{k-1})-D_{\Phi}(x_{k},w_{k-1})\right].
\end{equation*}
By rearranging the above equation, we get 
\begin{equation*}
    \frac{\psi}{\psi-1}D_{\Phi}(p,w_{k}) = \frac{1}{\psi-1}D_{\Phi}(p,w_{k-1})-\frac{1}{\psi-1}D_{\Phi}(x_{k},w_{k-1})+\frac{\psi}{\psi-1}D_{\Phi}(x_{k},w_{k})+D_{\Phi}(p,x_{k}). \tag{4.8}\label{4.8}
\end{equation*}
Using \eqref{2.3} and $\nabla \Phi(w_{k})=\frac{\psi-1}{\psi}\nabla {\Phi}(x_{k})+\frac{1}{\psi}\nabla{\Phi}(w_{k-1})$ again, we conclude that 
\begin{align*}
    D_{\Phi}(x_{k},w_{k}) = \frac{1}{\psi} D_{\Phi}(x_{k},w_{k-1})-\frac{1}{\psi} D_{\Phi}(w_{k},w_{k-1})-\frac{\psi-1}{\psi} D_{\Phi}(w_{k},x_{k}). \tag{4.9}\label{4.9}
\end{align*}
Substituting \eqref{4.9} into \eqref{4.8} yields the result
\begin{align*}
    \left(1+\frac{1}{\psi-1}\right)D_{\Phi}(p,w_{k}) & = \frac{\psi}{\psi-1}\left(\frac{1}{\psi} D_{\Phi}(x_{k},w_{k-1})-\frac{1}{\psi} D_{\Phi}(w_{k},w_{k-1})-\frac{\psi-1}{\psi} D_{\Phi}(w_{k},x_{k})\right)\\
    &\quad+D_{\Phi}(p,x_{k})+\frac{1}{\psi-1}D_{\Phi}(p,w_{k-1})-\frac{1}{\psi-1}D_{\Phi}(x_{k},w_{k-1})\\
    & = D_{\Phi}(p,x_{k})+\frac{1}{\psi-1}D_{\Phi}(p,w_{k-1})-\frac{1}{\psi-1} D_{\Phi}(w_{k},w_{k-1})-D_{\Phi}(w_{k},x_{k})\tag{4.10}\label{4.10}
\end{align*}
Substituting the above equation into \eqref{4.7} yields
\begin{align*}
     D_{\Phi}(p, z_k)&\leq -\left(1-\frac{\mu\lambda_{k}}{\rho\lambda_{k+1}}\right) D_{\Phi}(y_k, w_k) -\left(1-\frac{\mu\lambda_{k}}{\rho\lambda_{k+1}}\right) D_{\Phi}(z_k, y_k)+D_{\Phi}(p, x_k)\\
     &\quad -\frac{1}{\psi-1}D_{\Phi}(p, w_k)+\frac{1}{\psi-1}D_{\Phi}(p, w_{k-1})-\frac{1}{\psi-1}D_{\Phi}(w_k, w_{k-1})-D_{\Phi}(w_{k},x_{k}).\tag{4.11}\label{4.11}
\end{align*}
In addition, from the definition of $\{x_{k+1}\}$ and \eqref{bregman_jensen}, we obtain 
\begin{align*}
    D_{\Phi}(p,x_{k+1}) &= D_{\Phi}(p,\nabla \Phi^{*} \left( (1 - \gamma) \nabla \Phi(x_k) + \gamma \nabla \Phi(z_k) \right))
    \leq (1-\gamma)D_{\Phi}(p,x_{k})+\gamma D_{\Phi}(p,z_{k}).\tag{4.12}\label{4.12}
\end{align*}
Similarly, according to \eqref{4.11} and \eqref{4.12}, it follows that 
\begin{align*}
    D_{\Phi}(p,x_{k+1})+\frac{\gamma}{\psi-1}D_{\Phi}(p,w_{k})&\leq  D_{\Phi}(p,x_{k})+\frac{\gamma}{\psi-1}D_{\Phi}(p,w_{k-1})-\frac{\gamma}{\psi-1}D_{\Phi}(w_k,w_{k-1})-\gamma D_{\Phi}(w_{k},x_{k})\\
    &\quad -\gamma\left(1-\frac{\mu\lambda_{k}}{\rho\lambda_{k+1}}\right) D_{\Phi}(y_k, w_k) -\gamma\left(1-\frac{\mu\lambda_{k}}{\rho\lambda_{k+1}}\right) D_{\Phi}(z_k, y_k).\tag{4.13}\label{4.13}
\end{align*}
Let $\Upsilon_{k} = D_{\Phi}(p,x_{k})+\frac{\gamma}{\psi-1}D_{\Phi}(p,w_{k-1})$ and 
\begin{equation*}
   \Theta_{k}=\frac{\gamma}{\psi-1}D_{\Phi}(w_k,w_{k-1})+\gamma D_{\Phi}(w_{k},x_{k})+\gamma\left(1-\frac{\mu\lambda_{k}}{\rho\lambda_{k+1}}\right) D_{\Phi}(y_k, w_k)+\gamma\left(1-\frac{\mu\lambda_{k}}{\rho\lambda_{k+1}}\right) D_{\Phi}(z_k, y_k). 
\end{equation*}
Hence, we derive that 
\begin{align*}
    \Upsilon_{k+1}\leq \Upsilon_{k}-\Theta_{k}. 
\end{align*}
According to Lemma \ref{l2.7}, it is not difficult to see that the limit of $\Theta_{k}\to 0$ as $k\to\infty$ and $\lim_{k\to\infty}\Upsilon_{k}$ exists. Therefore, 
\begin{equation*}
    \lim_{k\to\infty}D_{\Phi}(w_k,w_{k-1}) =  \lim_{k\to\infty}D_{\Phi}(w_{k},x_{k}) =   \lim_{k\to\infty}D_{\Phi}(y_k, w_k) =  \lim_{k\to\infty}D_{\Phi}(z_k, y_k) =0.   
\end{equation*}
which implies that $
\lim_{k \to \infty} \|w_k - w_{k-1}\| = 
\lim_{k \to \infty} \|w_k - x_k\| = 
\lim_{k \to \infty} \|w_k - y_k\| = 
\lim_{k \to \infty} \|z_k - y_k\| = 0$ 
and $\lim_{k\to\infty}D_{\Phi}(x_{k},w_{k-1})=0$. 
\vskip 2mm 
In addition, from the definition of $\{w_{k}\}$, we have 
\begin{align*}
    \Upsilon_k &= D_{\Phi}(p,x_{k})+\frac{\gamma}{\psi-1}D_{\Phi}(p,w_{k-1})\\
    & =\frac{\psi}{\psi-1}\left[D_{\Phi}(p,w_{k})-D_{\Phi}(x_{k},w_{k})\right]-\frac{1}{\psi-1}\left[D_{\Phi}(p,w_{k-1})-D_{\Phi}(x_{k},w_{k-1})\right]+\frac{\gamma}{\psi-1}D_{\Phi}(p,w_{k-1})\\
    &\leq \frac{\psi}{\psi-1}D_{\Phi}(p,w_{k})-\frac{\psi}{\psi-1}D_{\Phi}(x_{k},w_{k})+\frac{1}{\psi-1}D_{\Phi}(x_{k},w_{k-1})\\
    &\quad-\frac{\gamma}{\psi-1}D_{\Phi}(p,w_{k-1})+\frac{\gamma}{\psi-1}D_{\Phi}(p,w_{k-1})\\
    & = \frac{\psi}{\psi-1}D_{\Phi}(p,w_{k})-\frac{\psi}{\psi-1}D_{\Phi}(x_{k},w_{k})+\frac{1}{\psi-1}D_{\Phi}(x_{k},w_{k-1}).\\
\end{align*}
Because of $\lim_{k\to\infty}D_{\Phi}(x_{k},w_{k})=\lim_{k\to\infty}D_{\Phi}(x_{k},w_{k-1})=0$ and the limit of $\Upsilon_{k}$ exists, we obtain easily that the limit of $D_{\Phi}(p,w_{k})$ exists. Hence, it turns out that $\lim_{k \to \infty} D_\Phi(p, x_{k})$  exists and the sequence $\{x_{k}\}$ and $\{w_{k}\}$ are bounded. 
\vskip 2mm 
\textbf{Case II:} Assume that $\lambda_{k}\to 0$ as $k\to\infty$. Using Lemma \ref{l2.2}, by taking $\epsilon=\frac{1}{2}$, we conclude that 
\begin{equation*}
    \|F(w_{k})-F(y_{k})\|\leq M\|w_{k}-y_{k}\|+\frac{1}{2}. \tag{4.14}\label{4.14}
\end{equation*}
According to $\lim_{k\to\infty}\lambda_{k} = 0$, there exists $K$ such that for all $k>K$, $\lambda_{k}<\frac{\mu}{2M}$. Next, we will consider the following two situations.
\vskip 2mm 
\textbf{Case A:} If $\lambda_{k+1}\geq \lambda_{k}$. According to the definition of $\lambda_{k}$, we get 
\begin{equation*}
    \lambda_{k+1}\leq \frac{\mu \|w_{k}-y_{k}\|}{  \|F(w_{k})-F(y_{k})\|}. 
\end{equation*}
Hence, it follows from \eqref{4.6} that 
\begin{align*}
    D_{\Phi}(p, z_k) 
    &\leq D_{\Phi}(p, w_k) - D_{\Phi}(y_k, w_k) - D_{\Phi}(z_k, y_k) + \lambda_k \langle z_k - y_k, F(w_k) - F(y_k) \rangle\\
    &\leq D_{\Phi}(p, w_k) - D_{\Phi}(y_k, w_k) - D_{\Phi}(z_k, y_k) + \lambda_k \| z_k - y_k\| \|F(w_k) - F(y_k)\|\\
    & \leq D_{\Phi}(p, w_k) - D_{\Phi}(y_k, w_k) - D_{\Phi}(z_k, y_k) + \frac{\mu\lambda_{k}}{\lambda_{k+1}} \| z_k - y_k\| \|w_k - y_k\|\\
    & \leq D_{\Phi}(p, w_k) -\left(1-\frac{\mu\lambda_{k}}{\rho\lambda_{k+1}}\right) D_{\Phi}(y_k, w_k) -\left(1-\frac{\mu\lambda_{k}}{\rho\lambda_{k+1}}\right) D_{\Phi}(z_k, y_k).\\
    & \leq D_{\Phi}(p, w_k) -\left(1-\frac{\mu\lambda_{k}}{\rho\lambda_{k}}\right) D_{\Phi}(y_k, w_k) -\left(1-\frac{\mu\lambda_{k}}{\rho\lambda_{k}}\right) D_{\Phi}(z_k, y_k).\\
    & \leq D_{\Phi}(p, w_k) -\left(1-\frac{\mu}{\rho}\right) D_{\Phi}(y_k, w_k) -\left(1-\frac{\mu}{\rho}\right) D_{\Phi}(z_k, y_k).
\end{align*}
The following proof is similar to the above \textbf{Case I}, so we omitted.
\vskip 2mm 
\textbf{Case B:} If $\lambda_{k+1}< \lambda_{k}$. Then 
\begin{equation*}
    \lambda_{k+1}= \frac{\mu \|w_{k}-y_{k}\|}{  \|F(w_{k})-F(y_{k})\|}. 
\end{equation*}
Hence, from \eqref{4.14}, we get 
\begin{equation*}
    \|\|F(w_{k})-F(y_{k})\|\leq \frac{M}{\mu}\lambda_{k+1}\|F(w_{k})-F(y_{k})\|+\frac{1}{2}.
\end{equation*}
According to $\lambda_{k+1}<\lambda_{k}<\frac{\mu}{2M}$. Hence, we have 
\begin{equation*}
    \|F(w_{k})-F(y_{k})\|\leq 1. 
\end{equation*}
Similarly, from \eqref{4.6} and  $D_{\Phi}(x,y)\geq \frac{\rho}{2}\|x-y\|^2\geq\frac{1}{2}\|x-y\|^2$, it follows that 
\begin{align*}
      D_{\Phi}(p, z_k) 
    &\leq D_{\Phi}(p, w_k) - D_{\Phi}(y_k, w_k) - D_{\Phi}(z_k, y_k) + \lambda_k \langle z_k - y_k, F(w_k) - F(y_k) \rangle\\
    &\leq D_{\Phi}(p, w_k) - D_{\Phi}(y_k, w_k) - D_{\Phi}(z_k, y_k) + \lambda_k \| z_k - y_k\| \|F(w_k) - F(y_k)\|\\
   & \leq D_{\Phi}(p, w_k) - D_{\Phi}(y_k, w_k) - D_{\Phi}(z_k, y_k)+\frac{1}{2}\lambda_{k}^2\|F(w_{k})-F(y_{k})\|^2+\frac{1}{2}\|z_{k}-y_{k}\|^2\\
    & \leq D_{\Phi}(p, w_k) - D_{\Phi}(y_k, w_k) - D_{\Phi}(z_k, y_k)+\frac{1}{2}\lambda_{k}^2\|F(w_{k})-F(y_{k})\|^2+\frac{1}{\rho}D_{\Phi}(z_{k},y_{k})\\
    & \leq D_{\Phi}(p, w_k) - D_{\Phi}(y_k, w_k) - (1-\frac{1}{\rho})D_{\Phi}(z_k, y_k)+\frac{1}{2}\lambda_{k}^2\|F(w_{k})-F(y_{k})\|^2\\
    & \leq D_{\Phi}(p, w_k)-(1-\frac{1}{\rho})D_{\Phi}(z_k, y_k)-\frac{\rho}{2}\|y_{k}-w_{k}\|^2+\frac{1}{2}\lambda_{k}^2\|F(w_{k})-F(y_{k})\|^2\\
     & \leq D_{\Phi}(p, w_k)-(1-\frac{1}{\rho})D_{\Phi}(z_k, y_k)-\frac{\rho}{2}\frac{\lambda_{k+1}^2}{\mu^2}\|F(y_{k})-F(w_{k})\|^2+\frac{1}{2}\lambda_{k}^2\|F(w_{k})-F(y_{k})\|^2\\
     & = D_{\Phi}(p, w_k)-(1-\frac{1}{\rho})D_{\Phi}(z_k, y_k)-\frac{\rho}{2}(\frac{1}{\mu^2}-1)\lambda_{k+1}^2\|F(y_{k})-F(w_{k})\|^2\\
     &\quad -\frac{\rho}{2}\lambda_{k+1}^2\|F(y_{k})-F(w_{k})\|^2+\frac{1}{2}\lambda_{k}^2\|F(w_{k})-F(y_{k})\|^2\\
     &= D_{\Phi}(p, w_k)-(1-\frac{1}{\rho})D_{\Phi}(z_k, y_k)-\frac{\rho}{2}(1-\mu^2)\frac{\lambda_{k+1}^2}{\mu^2}\|F(y_{k})-F(w_{k})\|^2+\frac{1}{2}\lambda_{k}^2-\frac{\rho}{2}\lambda_{k+1}^2\\
       &= D_{\Phi}(p, w_k)-(1-\frac{1}{\rho})D_{\Phi}(z_k, y_k)-\frac{\rho}{2}(1-\mu^2)\|y_{k}-w_{k}\|^2+\frac{\rho}{2}\lambda_{k}^2-\frac{\rho}{2}\lambda_{k+1}^2\\
        &= D_{\Phi}(p, w_k)-(1-\frac{1}{\rho})D_{\Phi}(z_k, y_k)-\frac{\rho}{2}(1-\mu^2)\|y_{k}-w_{k}\|^2 +\frac{\rho}{2}(\lambda_{k+1}+\lambda_{k})(\lambda_{k+1}-\lambda_{k})_{-}\\
         &= D_{\Phi}(p, w_k)-(1-\frac{1}{\rho})D_{\Phi}(z_k, y_k)-\frac{\rho}{2}(1-\mu^2)\|y_{k}-w_{k}\|^2 +\frac{\rho}{2}\frac{\mu}{M}(\lambda_{k+1}-\lambda_{k})_{-}\tag{4.15}\label{4.15}.
\end{align*}
Furthermore, according to \eqref{4.10} and the definition of $\{x_{k+1}\}$, we have 
\begin{align*}
    D_{\Phi}(p,z_{k})&\leq D_{\Phi}(p,x_{k})-\frac{1}{\psi-1}D_{\Phi}(p, w_k)+\frac{1}{\psi-1}D_{\Phi}(p, w_{k-1})-\frac{1}{\psi-1}D_{\Phi}(w_k, w_{k-1})-D_{\Phi}(w_{k},x_{k})\\
    &\quad -(1-\frac{1}{\rho})D_{\Phi}(z_k, y_k)-\frac{\rho}{2}(1-\mu^2)\|y_{k}-w_{k}\|^2 +\frac{\rho}{2}\frac{\mu}{M}(\lambda_{k+1}-\lambda_{k})_{-}, 
\end{align*}
and 
\begin{align*}
    D_{\Phi}(p,x_{k+1})+\frac{\gamma}{\psi-1}D_{\Phi}(p,w_{k})&\leq  D_{\Phi}(p,x_{k})+\frac{\gamma}{\psi-1}D_{\Phi}(p,w_{k-1})-\frac{\gamma}{\psi-1}D_{\Phi}(w_{k},w_{k-1})-\gamma D_{\Phi}(w_{k},x_{k})\\
    &\quad -\gamma(1-\frac{1}{\rho})D_{\Phi}(z_k, y_k)-\gamma\frac{\rho}{2}(1-\mu^2)\|y_{k}-w_{k}\|^2 +\frac{\gamma \rho}{2}\frac{\mu}{M}(\lambda_{k+1}-\lambda_{k})_{-}.\tag{4.16}\label{4.16} 
\end{align*}
Similarly, Let $\hat{\Upsilon}_{k} = D_{\Phi}(p,x_{k})+\frac{\gamma}{\psi-1}D_{\Phi}(p,w_{k-1})$, $\theta_{k} = \frac{\gamma \rho}{2}\frac{\mu}{M}(\lambda_{k+1}-\lambda_{k})_{-}$ and 
\begin{align*}
    \hat{\Theta}_{k} = \frac{\gamma}{\psi-1}D_{\Phi}(w_{k},w_{k-1})+\gamma D_{\Phi}(w_{k},x_{k})+\gamma(1-\frac{1}{\rho})D_{\Phi}(z_k, y_k)+\gamma\frac{\rho}{2}(1-\mu^2)\|y_{k}-w_{k}\|^2.
\end{align*}
Hence, we have 
\begin{equation*}
    \hat{\Upsilon}_{k+1}\leq \hat{\Upsilon}_{k}-\hat{\Theta}_{k}+\theta_{k}. 
\end{equation*}
Due to $\sum_{k=0}^\infty (\lambda_{k+1} - \lambda_k)_-<\infty$, we conclude that the limit of $\hat{\Theta}_{k}\to 0$ as $k\to\infty$ and $\lim_{k\to\infty}\hat{\Upsilon}_{k}$ exists.
The following proof is similar to the above \textbf{Case I}, so we omit.
\vskip 2mm 
Hence, combining the above discussion, conclusions (1) and (2) still hold, which completes the proof. 
\end{proof}
\begin{lemma}\label{l4.3}
    Suppose that $\lambda_k \to 0$ as $k \to \infty$, and define the index set $\mathcal{I} := \{k \in \mathbb{N} : \lambda_{k+1} < \lambda_k\}$. Let $\{k_i\}_{i \geq 1}$ denote the strictly increasing sequence of indices in $\mathcal{I}$. Then, for any $p \in \Omega_D$ and for all $k_i \in \mathcal{I}$, the following assertions hold:
    \vskip 2mm 
    (i) $\displaystyle\lim_{i \to \infty} \frac{\|w_{k_i} - y_{k_i}\|}{\lambda_{k_i}} = 0$;
    \vskip 2mm 
    (ii) $\displaystyle\lim_{i \to \infty} \frac{1}{\lambda_{k_i}} \left( D_{\Phi}(p, x_{k_{i+1}}) - D_{\Phi}(p, x_{k_i}) + \frac{\gamma}{\psi - 1} \left( D_{\Phi}(p, w_{k_i}) - D_{\Phi}(p, w_{k_{i-1}}) \right) \right) = 0$.
\end{lemma}

\begin{proof}
    Since $\lambda_k \to 0$ as $k \to \infty$, the set $\mathcal{I} := \{k \in \mathbb{N} : \lambda_{k+1} < \lambda_k\}$ is infinite. Hence, the sequence $\{k_i\}$ is well-defined and strictly increasing. By the definition of $\lambda_k$, we have
    \[
        \lambda_{k_{i+1}} = \frac{\mu \|w_{k_i} - y_{k_i}\|}{\|F(w_{k_i}) - F(y_{k_i})\|} < \lambda_{k_i}, \quad \text{for all } i \geq 1,
    \]
    which implies
    \[
        \frac{\|w_{k_i} - y_{k_i}\|}{\lambda_{k_i}} \leq \frac{1}{\mu} \|F(w_{k_i}) - F(y_{k_i})\|.
    \]
    Since it follows from Lemma~\ref{l4.2} (part (1)) that $\|w_k - y_k\| \to 0$, and $F$ is uniformly continuous, we obtain
    \[
        \lim_{i \to \infty} \|F(w_{k_i}) - F(y_{k_i})\| = 0,
    \]
    and hence,
    \[
        \lim_{i \to \infty} \frac{\|w_{k_i} - y_{k_i}\|}{\lambda_{k_i}} = 0.
    \]
    This completes the proof of part (i).

    \vskip 2mm
    Regarding part (ii), we recall inequality~\eqref{4.15}, which without any rescaling on $\lambda_k$, yields
   \begin{align*}
    D_{\Phi}(p,x_{k+1})+\frac{\gamma}{\psi-1}D_{\Phi}(p,w_{k})&\leq  D_{\Phi}(p,x_{k})+\frac{\gamma}{\psi-1}D_{\Phi}(p,w_{k-1})-\frac{\gamma}{\psi-1}D_{\Phi}(w_{k},w_{k-1})-\gamma D_{\Phi}(w_{k},x_{k})\\
    &\quad -\gamma(1-\frac{1}{\rho})D_{\Phi}(z_k, y_k)-\gamma\frac{\rho}{2}(1-\mu^2)\|y_{k}-w_{k}\|^2 +\frac{\rho}{2}\lambda_{k}^2-\frac{\rho}{2}\lambda_{k+1}^2\\
    &\leq D_{\Phi}(p,x_{k})+\frac{\gamma}{\psi-1}D_{\Phi}(p,w_{k-1})+\frac{\rho}{2}\lambda_{k}^2, 
\end{align*}
    where the last inequality follows by omitting nonpositive terms. Taking differences at the indices $\{k_i\}$ and dividing by $\lambda_{k_i}$, we obtain
    \[
        \frac{1}{\lambda_{k_i}} \left( D_{\Phi}(p, x_{k_{i+1}}) - D_{\Phi}(p, x_{k_i}) + \frac{\gamma}{\psi - 1} (D_{\Phi}(p, w_{k_i}) - D_{\Phi}(p, w_{k_{i-1}})) \right) \leq \frac{\rho}{2} \lambda_{k_i}.
    \]
    Since $\lambda_{k_i} \to 0$ as $i \to \infty$, the right-hand side converges to zero. Therefore,
    \[
        \lim_{i \to \infty} \frac{1}{\lambda_{k_i}} \left( D_{\Phi}(p, x_{k_{i+1}}) - D_{\Phi}(p, x_{k_i}) + \frac{\gamma}{\psi - 1} (D_{\Phi}(p, w_{k_i}) - D_{\Phi}(p, w_{k_{i-1}})) \right) = 0,
    \]
    which concludes the proof of part (ii).
\end{proof}
\begin{proposition}\label{p4.4}
    Suppose that assumptions \textbf{D1}–\textbf{D4} hold. Then there exists a weak cluster point \( p \) of the sequence \( \{y_k\} \) such that \( p \in \Omega_D \) or \( F(p) = 0 \).
\end{proposition}

\begin{proof}
    The proof proceeds in a manner analogous to that of Proposition~\ref{p3.5}, with the primary distinction arising when the stepsizes \( \lambda_k \to 0 \) as \( k \to \infty \). In such a case, we can only assert the existence of one weak cluster point \( p \) of \( \{y_k\} \) satisfying \( p \in \Omega_D \) or \( F(p) = 0 \), rather than the stronger conclusion in Proposition~\ref{p3.5} which guarantees this property for all weak cluster points.
    
    \vskip 2mm
    \noindent
    \textbf{Case 1:} \( \lambda_k \to \lambda > 0 \) as \( k \to \infty \). In this situation, the convergence \( \lambda_{k_i} \to \lambda > 0 \) also holds for any subsequence \( \{k_i\} \). Following the same argument as in Proposition~\ref{p3.5}, we conclude that every weak cluster point \( p \) of \( \{y_k\} \) satisfies \( p \in \Omega_D \) or \( F(p) = 0 \).

    \vskip 2mm
    \noindent
    \textbf{Case 2:} \( \lambda_k \to 0 \) as \( k \to \infty \). In this case, define the index set
    \[
        \mathcal{I} := \{k \in \mathbb{N} : \lambda_{k+1} < \lambda_k\}.
    \]
    Since \( \lambda_k \to 0 \), the set \( \mathcal{I} \) is infinite. Denote the strictly increasing sequence of indices in \( \mathcal{I} \) by \( \{k_i\} \). The boundedness of \( \{x_{k_i}\} \) implies the existence of a weakly convergent subsequence. Without loss of generality, assume \( x_{k_i} \rightharpoonup p \) as \( i \to \infty \). Moreover, by Lemma~\ref{l4.2} (part(1)), we also have \( y_{k_i} \rightharpoonup p \) and \( p \in K \).
    
    \vskip 1mm
    The essential difference in this case lies in the counterpart of Case 2 in Proposition~\ref{p3.5}. Recall that  the following inequality:
    \[
        \frac{1}{\lambda_{k_i}} \langle w_{k_i} - y_{k_i}, z - y_{k_i} \rangle - \langle F(w_{k_i}) - F(y_{k_i}), z - y_{k_i} \rangle \leq \langle F(y_{k_i}), z - y_{k_i} \rangle, \quad \forall z \in K.
    \]
    Since \( \|w_{k_i} - y_{k_i}\| \to 0 \) and \( F \) is uniformly continuous on bounded sets, it follows that
    \[
        \|F(w_{k_i}) - F(y_{k_i})\| \to 0.
    \]
    Combined with \( \frac{\|w_{k_i} - y_{k_i}\|}{\lambda_{k_i}} \to 0 \), the left-hand side of the inequality tends to 0 as \( i \to \infty \). Hence, taking limits, we obtain:
    \[
        0 \leq \liminf_{i \to \infty} \langle F(y_{k_i}), z - y_{k_i} \rangle \leq \limsup_{i \to \infty} \langle F(y_{k_i}), z - y_{k_i} \rangle < +\infty.
    \]
    The remaining steps follow similarly to those in the proof of Proposition~\ref{p3.5}, leading to the conclusion that the weak limit point \( p \) satisfies \( p \in \Omega_D \) or \( F(p) = 0 \). For brevity, we omit the details here.
\end{proof}
\begin{theorem}
Assume that \(F(p) \neq 0\) for all \(p \in \Omega \setminus \Omega_D\). Under assumptions \textbf{D1}–\textbf{D4} and \textbf{C5}, the sequence \(\{x_k\}\) generated by Algorithm 1 converges strongly to a point \(p \in \Omega_D\).
\end{theorem}

\begin{proof}
The proof is divided into two cases:

\vskip 2mm
\textbf{Case I:} \(\lambda_k \to \lambda > 0\) as \(k \to \infty\).

By Lemma \ref{l4.2}, there exists a subsequence \(\{x_{k_i}\}\) of \(\{x_k\}\) such that
$x_{k_i} \rightharpoonup p \in \Omega$.
Since \(F(p) \neq 0\) for all \(p \in \Omega \setminus \Omega_D\), it follows that \(p \in \Omega_D\). Given \(\lim_{i \to \infty} \|x_{k_i} - y_{k_i}\| = 0\), we also have
$y_{k_i} \rightharpoonup p \in \Omega_D$.
Because \(\{y_{k_i}\} \subset K\), assumption \textbf{C5} guarantees the existence of constants \(K' > K\) and \(c' > 0\) such that for all \(k > K'\),
\begin{equation*}
\langle F(y_{k_i}), y_{k_i} - p \rangle \geq c' \|y_{k_i} - p\|^{2 + \epsilon}\tag{4.17}\label{eq:4.17}.
\end{equation*}

Combining inequalities \eqref{4.3}, \eqref{4.7}, and \eqref{eq:4.17}, we deduce
\[
\begin{aligned}
D_{\Phi}(p, z_{k_i}) &\leq D_{\Phi}(p, w_{k_i}) - \left(1 - \frac{\mu \lambda_{k_i}}{\rho \lambda_{k_{i+1}}}\right) D_{\Phi}(y_{k_i}, w_{k_i}) - \left(1 - \frac{\mu \lambda_{k_i}}{\rho \lambda_{k_{i+1}}}\right) D_{\Phi}(z_{k_i}, y_{k_i}) - \lambda_{k_i} \langle F(y_{k_i}), y_{k_i} - p \rangle \\
&\leq D_{\Phi}(p, w_{k_i}) - \left(1 - \frac{\mu \lambda_{k_i}}{\rho \lambda_{k_{i+1}}}\right) D_{\Phi}(y_{k_i}, w_{k_i}) - \left(1 - \frac{\mu \lambda_{k_i}}{\rho \lambda_{k_{i+1}}}\right) D_{\Phi}(z_{k_i}, y_{k_i}) - c' \lambda_{k_i} \|y_{k_i} - p\|^{2 + \epsilon}.
\end{aligned}
\]
From inequality \eqref{4.13}, we obtain
\[
\begin{aligned}
D_{\Phi}(p, x_{k+1}) + \frac{\gamma}{\psi - 1} D_{\Phi}(p, w_k) &\leq D_{\Phi}(p, x_k) + \frac{\gamma}{\psi - 1} D_{\Phi}(p, w_{k-1}) - \frac{\gamma}{\psi - 1} D_{\Phi}(w_k, w_{k-1}) - \gamma D_{\Phi}(w_k, x_k) \\
&\quad - \gamma \left(1 - \frac{\mu \lambda_k}{\rho \lambda_{k+1}}\right) \bigl(D_{\Phi}(y_k, w_k) + D_{\Phi}(z_k, y_k)\bigr) - c' \lambda_k \|y_k - p\|^{2 + \epsilon},
\end{aligned}
\]
which can be rearranged as
\[
\begin{aligned}
c' \lambda_{k_i} \|y_{k_i} - p\|^{2 + \epsilon} &\leq \bigl[D_{\Phi}(p, x_k) - D_{\Phi}(p, x_{k+1})\bigr] + \frac{\gamma}{\psi - 1} \bigl[D_{\Phi}(p, w_{k-1}) - D_{\Phi}(p, w_k)\bigr] - \frac{\gamma}{\psi - 1} D_{\Phi}(w_k, w_{k-1}) \\
&\quad - \gamma \left(1 - \frac{\mu \lambda_k}{\rho \lambda_{k+1}}\right) \bigl(D_{\Phi}(y_k, w_k) + D_{\Phi}(z_k, y_k)\bigr) - \gamma D_{\Phi}(w_k, x_k).
\end{aligned}
\]
Noting that \(D_{\Phi}(p, w_k)\) and \(D_{\Phi}(p, x_k)\) have limits by Lemma \ref{l4.2}, and
\[
\lim_{k \to \infty} D_{\Phi}(w_k, w_{k-1}) = \lim_{k \to \infty} D_{\Phi}(w_k, x_k) = \lim_{k \to \infty} D_{\Phi}(y_k, w_k) = \lim_{k \to \infty} D_{\Phi}(z_k, y_k) = 0,
\]
it follows that
\[
\lim_{i \to \infty} \|y_{k_i} - p\| = 0.
\]
Therefore,
\[
\|x_{k_i} - p\| \leq \|x_{k_i} - y_{k_i}\| + \|y_{k_i} - p\| \to 0 \quad \text{as } i \to \infty.
\]
Since \(\lim_{k \to \infty} D_{\Phi}(p, x_k)\) exists, this implies the entire sequence converges strongly:
\[
\|x_k - p\| \to 0 \quad \text{as } k \to \infty.
\]

\vskip 2mm
\textbf{Case II:} \(\lambda_k \to 0\) as \(k \to \infty\).

In this case, the set
\[
\mathcal{I} := \{ k \in \mathbb{N} : \lambda_{k+1} < \lambda_k \}
\]
is infinite. Let \(\{k_i\}\) denote the strictly increasing sequence enumerating \(\mathcal{I}\). Then \(\lambda_{k_i} \to 0\) and \(\lambda_{k_{i+1}} \leq \lambda_{k_i}\) as \(i \to \infty\).
By Proposition \ref{p4.4} and the assumption \(F(p) \neq 0\) for all \(p \in \Omega \setminus \Omega_D\), we have $x_{k_i} \rightharpoonup p \in \Omega_D$.
Since \(\|x_{k_i} - y_{k_i}\| \to 0\), it follows that
$y_{k_i} \rightharpoonup p \in \Omega_D$.
Hence, assumption \textbf{C5} ensures that inequality \eqref{eq:4.17} holds.
\vskip 2mm 
Recalling inequalities \eqref{4.3} and \eqref{4.15}, and noting no scaling is applied to \(\lambda_k\), we have
\[
D_{\Phi}(p, z_{k_i}) \leq D_{\Phi}(p, w_{k_i}) - \left(1 - \frac{1}{\rho}\right) D_{\Phi}(z_{k_i}, y_{k_i}) - \frac{\rho}{2} (1 - \mu^2) \|y_{k_i} - w_{k_i}\|^2 + \frac{\rho}{2} \lambda_{k_i}^2 - \frac{\rho}{2} \lambda_{k_{i+1}}^2 - \lambda_{k_i} \langle F(y_{k_i}), y_{k_i} - p \rangle.
\]

Similarly, by \eqref{4.16}, we get 
\begin{align*}
    D_{\Phi}(p,x_{k_{i+1}})+\frac{\gamma}{\psi-1}D_{\Phi}(p,w_{k_i})&\leq  D_{\Phi}(p,x_{k_i})+\frac{\gamma}{\psi-1}D_{\Phi}(p,w_{k_{i-1}})-\frac{\gamma}{\psi-1}D_{\Phi}(w_{k_i},w_{k_{i-1}})-\gamma D_{\Phi}(w_{k_i},x_{k_i})\\
    &\quad -\gamma(1-\frac{1}{\rho})D_{\Phi}(z_{k_i}, y_{k_i})-\gamma\frac{\rho}{2}(1-\mu^2)\|y_{k_i}-w_{k_i}\|^2 \\
    &\quad+\frac{\rho}{2}\lambda_{k_i}^2-\frac{\rho}{2}\lambda_{k_{i+1}}^2-\lambda_{k_i}\langle F(y_{k_i}), y_{k_i}-p\rangle\\
    &\leq D_{\Phi}(p,x_{k_i})+\frac{\gamma}{\psi-1}D_{\Phi}(p,w_{k_{i-1}})+\frac{\rho}{2}\lambda_{k_i}^2-c'\lambda_{k_i}\|y_{k_i}-p\|^{2+\epsilon},
\end{align*}
which can be written as 
\begin{align*}
    c'\|y_{k_i}-p\|^{2+\epsilon}&\leq\frac{1}{\lambda_{k_i}}([D_{\Phi}(p,x_{k_i})-D_{\Phi}(p,x_{k_{i+1}})]+\frac{\gamma}{\psi-1}\left(D_{\Phi}(p,w_{k_{i-1}})-D_{\Phi}(p,w_{k_i})\right))+\frac{\rho}{2}\lambda_{k_i}.
\end{align*}
Since $\lambda_{k_i}\to 0$ as $i\to\infty$ and Lemma \ref{l4.3} (part (ii))
 \[
        \lim_{i \to \infty} \frac{1}{\lambda_{k_i}} \left( D_{\Phi}(p, x_{k_{i+1}}) - D_{\Phi}(p, x_{k_i}) + \frac{\gamma}{\psi - 1} (D_{\Phi}(p, w_{k_i}) - D_{\Phi}(p, w_{k_{i-1}})) \right) = 0,
    \]
we derive that $\lim_{i\to\infty}\|y_{k_i}-p\|=0$. Hence, 
\begin{equation*}
    \|x_{k_i}-p\|\leq \|x_{k_i}-y_{k_i}\|+\|y_{k_i}-p\|\to 0, \quad i\to\infty. 
\end{equation*}
Consequently, by the fact that $\lim_{k\to\infty}D_{\Phi}(p,x_{k})$ exists, we have 
\begin{equation*}
    \|x_{k}-p\|\to 0, \quad k\to\infty. 
\end{equation*}

This completes the proof.
\end{proof}

\begin{theorem}[Ergodic Convergence of Discrete Iterates]
Assume that \(F(p) \neq 0\) for all \(p \in \Omega \setminus \Omega_D\). Assume that assumptions \textbf{D1}–\textbf{D4} and \textbf{C5} hold, and let \( \{s_k\}_{k \geq 1} \) be a sequence of strictly positive weights satisfying
\[
S_N := \sum_{k=1}^N s_k \to +\infty \quad \text{as } N \to \infty.
\]
Then, the (weighted) ergodic averages
\[
\bar{x}_N := \frac{1}{S_N} \sum_{k=1}^N s_k x_k
\]
converge strongly to \( x^* \) as \( N \to \infty \). In particular, by choosing \( s_k = 1 \), the unweighted Cesàro average
\[
\bar{x}_N = \frac{1}{N} \sum_{k=1}^N x_k
\]
also converges strongly to \( x^* \).
\end{theorem}

\begin{proof}
Fix any \( \varepsilon > 0 \). Since \( x_k \to x^* \) strongly, there exists \( K \in \mathbb{N} \) such that
\[
\|x_k - x^*\| < \frac{\varepsilon}{2}, \quad \forall k \geq K.
\]
Note that the boundedness of \( \{x_k\} \) implies there exists \( M > 0 \) such that
\[
\|x_k - x^*\| \leq M, \quad \forall k \in \mathbb{N}.
\]
We now decompose
\[
\bar{x}_N - x^* = \frac{1}{S_N} \sum_{k=1}^N s_k (x_k - x^*) = \frac{1}{S_N} \sum_{k=1}^{K-1} s_k (x_k - x^*) + \frac{1}{S_N} \sum_{k=K}^N s_k (x_k - x^*).
\]
Taking norms and using the triangle inequality yields:
\[
\|\bar{x}_N - x^*\| \leq \frac{1}{S_N} \sum_{k=1}^{K-1} s_k \|x_k - x^*\| + \frac{1}{S_N} \sum_{k=K}^N s_k \|x_k - x^*\|.
\]
For the first term, since \( s_k > 0 \) and \( \|x_k - x^*\| \leq M \),
\[
\frac{1}{S_N} \sum_{k=1}^{K-1} s_k \|x_k - x^*\| \leq \frac{M \sum_{k=1}^{K-1} s_k }{S_N} \to 0 \quad \text{as } N \to \infty,
\]
because \( S_N \to +\infty \).

For the second term, using \( \|x_k - x^*\| < \frac{\varepsilon}{2} \) for all \( k \geq K \),
\[
\frac{1}{S_N} \sum_{k=K}^N s_k \|x_k - x^*\| < \frac{\varepsilon}{2} \cdot \frac{ \sum_{k=K}^N s_k }{ S_N } \leq \frac{\varepsilon}{2}.
\]
Combining the above, for sufficiently large \( N \),
\[
\|\bar{x}_N - x^*\| < 0 + \frac{\varepsilon}{2} = \frac{\varepsilon}{2} < \varepsilon.
\]
Since \( \varepsilon > 0 \) was arbitrary, we conclude
\[
\bar{x}_N \to x^* \quad \text{strongly as } N \to \infty, 
\]
which completes the proof.
\end{proof}

\section{Numerical experiments}\label{sec:experiments}
\vskip 2mm 
\textbf{Example 5.1} 
Consider the variational inequality problem over the box constraint
\[
C = \{ x \in \mathbb{R}^3 : -5 \leq x_i \leq 5, \; i=1,2,3 \},
\]
with the operator
\[
F(x) = \bigl( e^{ - \| x \|^2 } + q \bigr) M x,
\]
where \( q = 0.2 \) and
\[
M =
\begin{bmatrix}
1 & 0 & -1 \\ 
0 & 1.5 & 0 \\ 
-1 & 0 & 2
\end{bmatrix}.
\]
Since \( e^{ - \| x \|^2 } + q > 0 \) for all \( x \), the solution satisfies
$M x^* = 0$. It is easy verify that the exact solution is
$x^* = [0; 0; 0]$.
As mentioned in \cite{a22},  the operator $F$ is continuous and strongly pseudomonotone with constant  $\eta = q\cdot\lambda_{\min}\approx0.0764$, where $\lambda_{\min}$ is the smallest eigenvalue of $M$, and Lipschtz continuous with constant $L\approx5.0679$. 
The initial point is chosen as \textbf{Case 1}:$x_0 = (-5,\ 4,\ 7)^T$; \textbf{Case 2}:$x_0 = (-4,\ 3,\ 5)^T$. We compare the performance of four different strategies for the stepsize $\lambda\in\{0.05, 0,1,0.2\}$.
\vskip 2mm 
Figure~\ref{fig:comparison} plots the component trajectories $x_1(t), x_2(t), x_3(t)$ under all three strategies. It is evident that the selection of $\lambda$ significantly affects the rate of convergence. Specifically, as the values of $\lambda$ decrease, the convergence of the trajectories deteriorates.
\vskip 2mm 
Figure~\ref{fig:energy_decay} illustrates the decay of the Lyapunov energy function
\[
V(t) = \frac{1}{2} \| x(t) - x^* \|^2
\]
under different fixed step sizes \(\lambda \in \{0.05, 0.1, 0.2\}\) within the proposed continuous-time forward-backward-forward (FBF) dynamical system. It is observed that the energy function decreases monotonically towards zero, confirming the stability and convergence of the system. Furthermore, larger step sizes (\(\lambda = 0.2\)) lead to a significantly faster decay of the energy function, while smaller step sizes (\(\lambda = 0.05\)) exhibit slower convergence.

\begin{figure}[ht]
\centering
\begin{minipage}{0.45\linewidth}
    \centering
    \includegraphics[width=\linewidth]{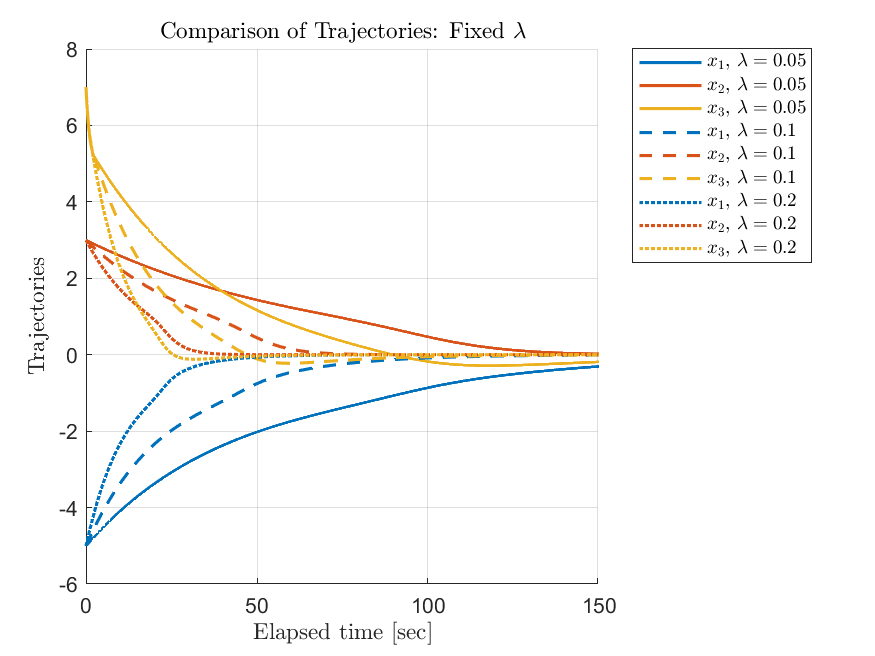}
\end{minipage} \hfill
\begin{minipage}{0.45\linewidth}
    \centering
    \includegraphics[width=\linewidth]{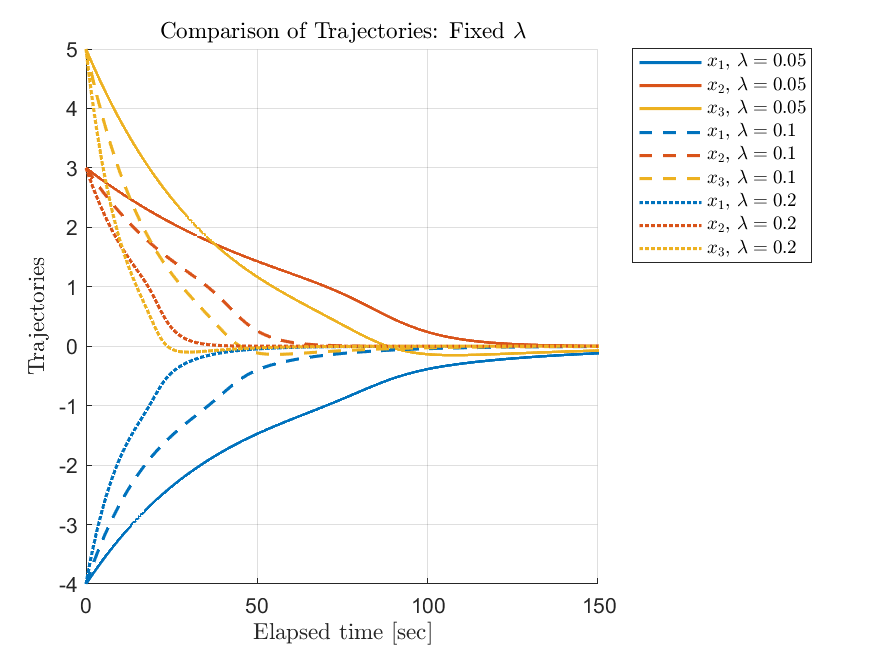}
\end{minipage}
\caption{The convergence trajectories of Example 5.1; Left:Case 1; Right: Case 2.}
\label{fig:comparison}
\end{figure}
\vskip 2mm
\begin{figure}[ht]
\centering
\begin{minipage}{0.45\linewidth}
    \centering
    \includegraphics[width=\linewidth]{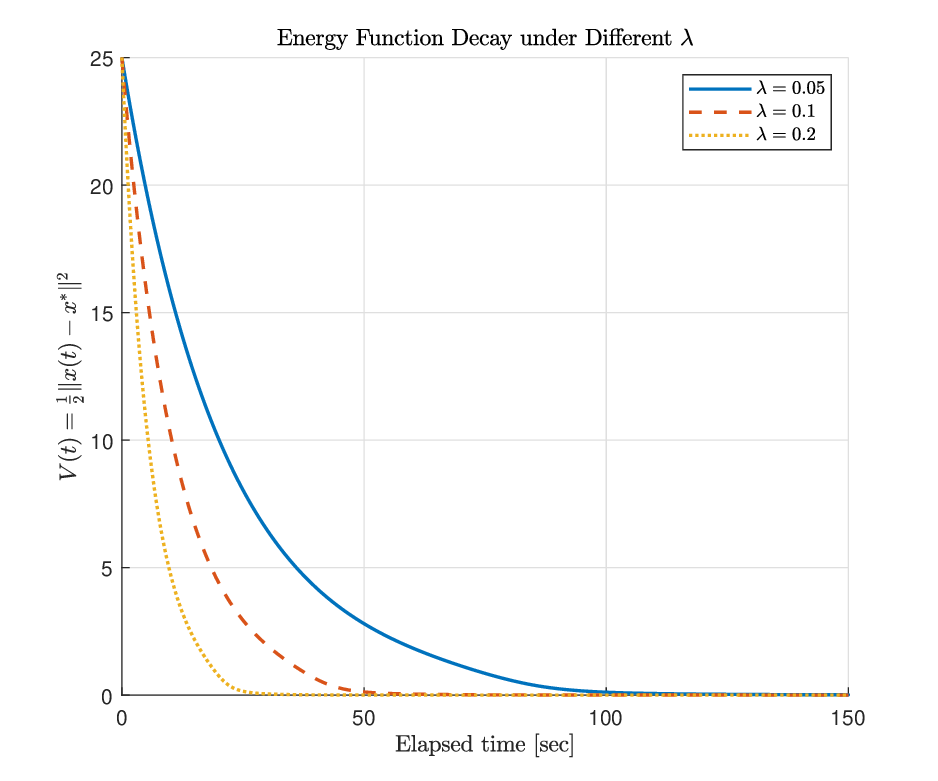}
\end{minipage} \hfill
\begin{minipage}{0.45\linewidth}
    \centering
    \includegraphics[width=\linewidth]{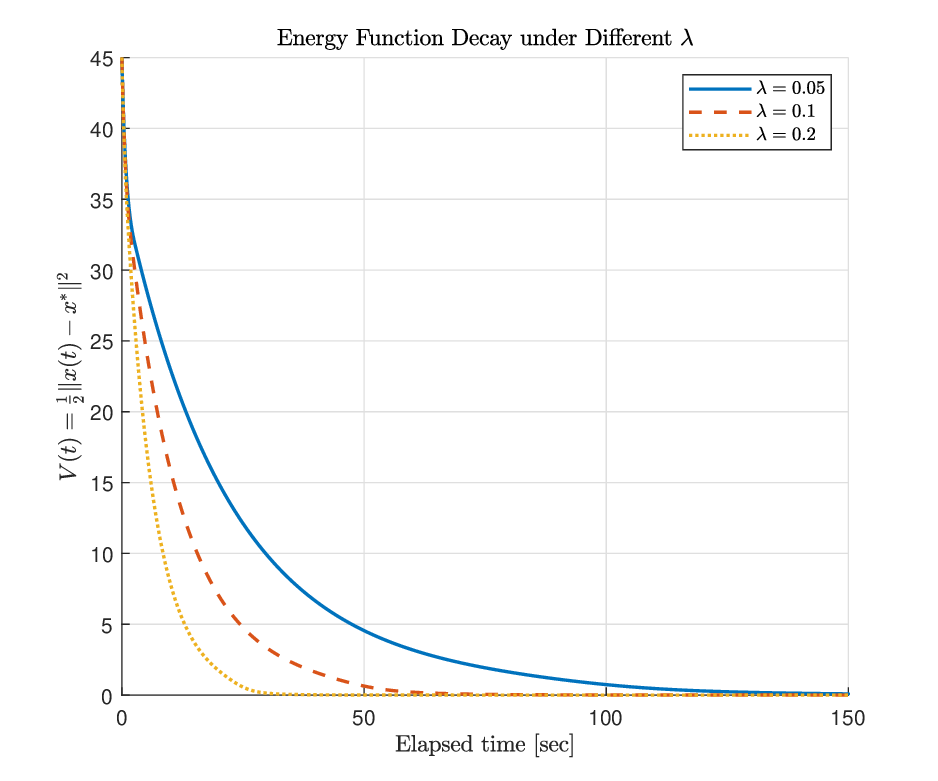}
\end{minipage}
\caption{Decay of the Lyapunov energy function under different step sizes \(\lambda\) of Example 5.1;; Left:Case 1; Right: Case 2.}
\label{fig:energy_decay}
\end{figure}
\vskip 2mm

\textbf{Example 5.2:}\quad In this numerical experiment, we revisit a classical problem originally introduced in \cite{a29}, which is formulated in the context of infinite-dimensional Hilbert spaces. Specifically, we let $E = \ell^2 := \{x=(x_1,x_2,\dots): \sum_{i=1}^\infty |x_i|^2 < \infty\}$, the space of square-summable sequences. Given two real constants $a,b$ satisfying $0 < \frac{b}{2} < a$, we define the constraint set $K := \{x \in \ell^2 : \|x\| \leq a\}$ and the operator $F : \ell^2 \to \ell^2$ by $F(x) = (b - \|x\|) x$. It can be verified that this operator is quasi-monotone and Lipschitz continuous, with the set of solutions $\Omega_D = \{0\}$.
\vskip 2mm 
We compare the performance of our proposed Algorithm 1 against several existing iterative schemes, namely Algorithm 1 from Oyewole et al.~\cite{a30}, Algorithm 1 from Malitsky~\cite{a26}. All algorithms are tested with the squared norm function $f(x) = \frac{1}{2} \|x\|^2$, except when otherwise specified.

\vskip 2mm
\noindent The parameter settings used in each method are as follows:
\vskip 2mm
$\bullet$ For our Algorithm 1: we choose $\lambda_1 = 0.15$, $\mu = 0.8$, $\gamma = 0.9$, and $\psi =\frac{1+\sqrt{5}}{2}$. 
\vskip 2mm
$\bullet$ For Oyewole et al.~\cite{a30} (Algorithm 1): set $\lambda_1 = 0.15$, $\mu = 0.8$, $\delta_k = \frac{1}{k^{1.1}}$, $\beta_k = 1+\delta_k$ and $\psi =\frac{1+\sqrt{5}}{2}$.
\vskip 2mm
$\bullet$ For Malitsky~\cite{a26} (Algorithm 1): use $\lambda_1 = 0.15$ and $\psi =\frac{1+\sqrt{5}}{2}$.
\vskip 2mm
All algorithms are executed until the stopping criterion $E_{k}=\|x_{k+1} - x_k\| < 10^{-5}$ is met, or the maximum number of iterations reaches 1000. We consider the following four parameter settings for $(a, b)$:
\textbf{Case I:} $(a, b) = (2, 3)$; \textbf{Case II:} $(a, b) = (3, 5)$; \textbf{Case III:} $(a, b) = (4, 7)$; \textbf{Case IV:} $(a, b) = (5, 9)$. 
The corresponding numerical results, including iteration counts and computational performance, are summarized in Figure \ref{4-2} and Table \ref{Table-1}.
\begin{figure}[ht]\label{fig3}
\centering
\begin{minipage}{0.45\linewidth}
    \centering
    \includegraphics[width=\linewidth]{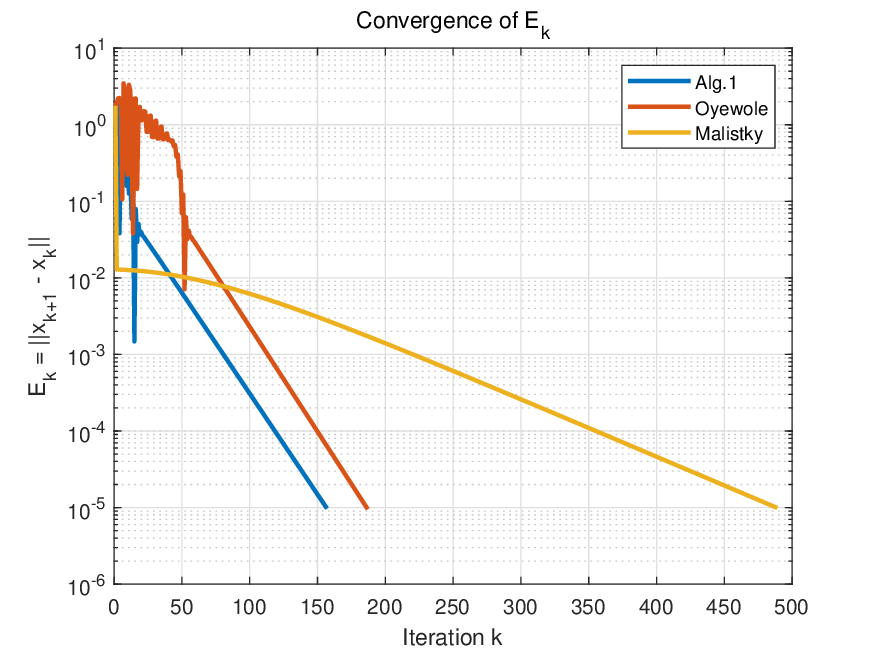}
    \label{fig:z1}
\end{minipage} \hfill
\begin{minipage}{0.45\linewidth}
    \centering
    \includegraphics[width=\linewidth]{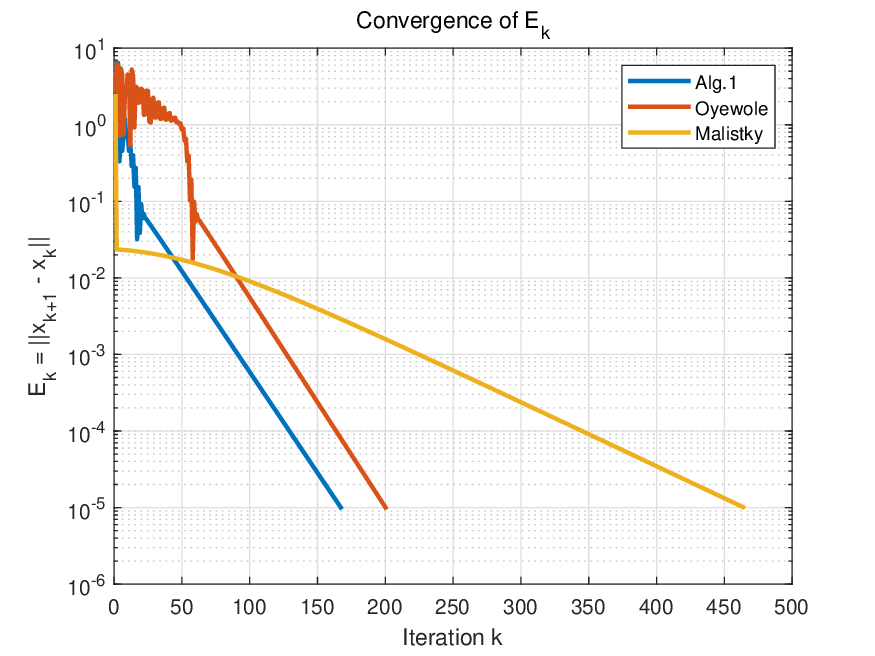}
    \label{fig:z2}
\end{minipage}
\centering
\begin{minipage}{0.45\linewidth}
    \centering
    \includegraphics[width=\linewidth]{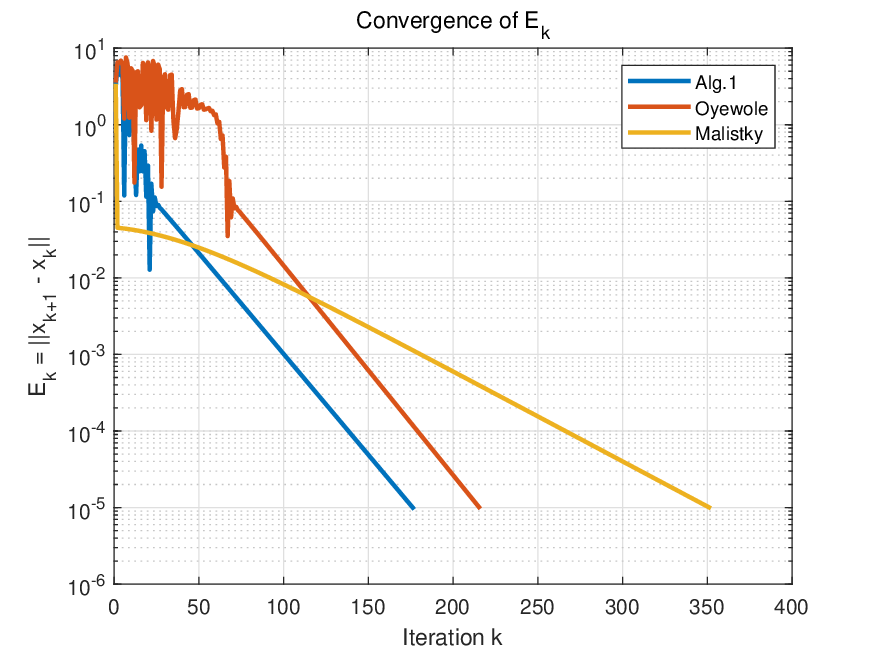}
    \label{fig:z3}
\end{minipage} \hfill
\begin{minipage}{0.45\linewidth}
    \centering
    \includegraphics[width=\linewidth]{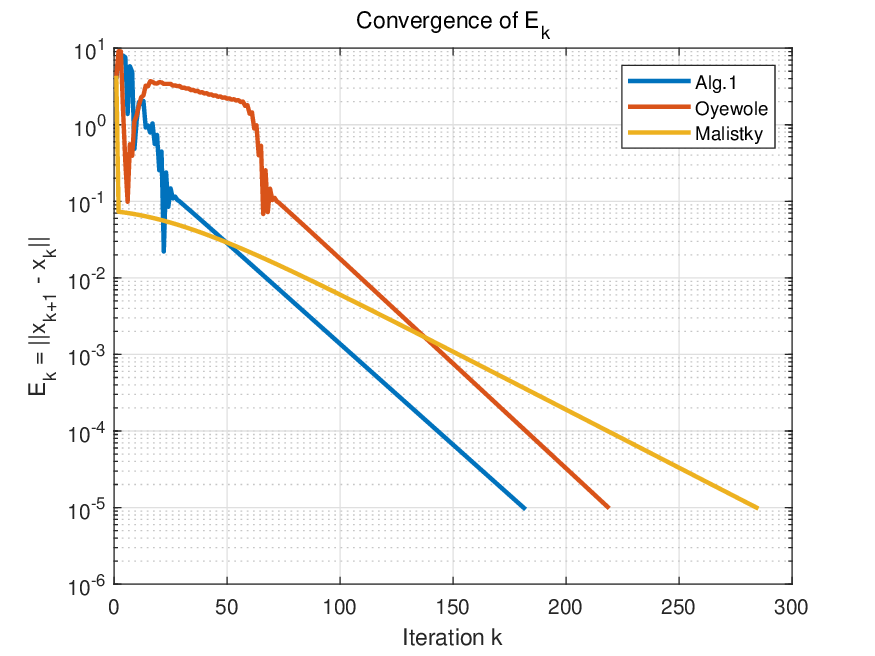}
    \label{fig:z4}
\end{minipage}
\caption{The numerical test results of Example 5.2; Top left:Case I; Top Right: Case II; Bottom left:Case III; Bottom right:Case IV}
\label{4-2}
\end{figure}

\begin{table}[ht]
\centering
\caption{Computation results for Example 5.2}
\renewcommand\arraystretch{1.5}
\begin{tabular}{>{\centering\arraybackslash}m{4.5cm}cccccc}
\toprule
\multirow{2}{*}{} & \multicolumn{3}{c}{No. of iterations} & \multicolumn{3}{c}{CPU time (sec)} \\
\cmidrule(lr){2-4} \cmidrule(lr){5-7}
 & Alg.1 & Oyewole et al. & Malitsky  & Alg.1 & Oyewole et al. & Malitsky \\
\midrule
Case I& 165 & 182 & 479 & 0.0826 & 0.0925 & 0.0969 \\
Case II & 173 & 202 & 420 & 0.0938 & 0.0109 & 0.1048 \\
Case III& 175 & 227 & 354 & 0.1050 & 0.1354 & 0.1869 \\
Case IV & 180 & 220 & 279 & 0.1250 & 0.1354 & 0.18695 \\
\bottomrule
\end{tabular}
\label{Table-1}
\end{table}

\vskip 2mm 
\textbf{Example 5.3} \quad To evaluate the performance of our Algorithm 1 with nonlipschitz operator, we consider a simple traffic equilibrium model with three parallel routes. Let \( x = (x_1, x_2, x_3)^\top \) denote the vector of path flow proportions, constrained to lie in the probability simplex \( \Delta := \{ x \in \mathbb{R}^3_+ : x_1 + x_2 + x_3 = 1 \} \). The cost (or delay) function on each path is defined as
\[
F(x) = \begin{bmatrix}
1 + x_1^q \\
1.5 + x_2^q \\
2 + x_3^q
\end{bmatrix}, \quad \text{with } q = 0.2,
\]
so that the cost of each route is an increasing function of its own flow. The traffic equilibrium corresponds to the solution of the following variational inequality:
\[
\text{Find } x^* \in \Delta \text{ such that } \langle F(x^*), y - x^* \rangle \geq 0, \quad \forall y \in \Delta.
\]
Specifically, \( \frac{\partial F_i}{\partial x_i} = q x_i^{q-1} > 0 \) for all \( x_i \in (0,1) \), while \( \frac{\partial F_i}{\partial x_j} = 0 \) for \( i \neq j \). Hence, the Jacobian of \( F \) is a positive diagonal matrix (but not symmetric), implying that \( F \) is not derived from any scalar potential. The mapping \( F \) is therefore not monotone in the classical sense, nor is it Lipschitz continuous, since \( x_i^{q-1} \to \infty \) as \( x_i \to 0^+ \) for \( q < 1 \). Nevertheless, \( F \) is pseudo-monotone on the simplex \( \Delta \): for any \( x, y \in \Delta \), if \( \langle F(x), y - x \rangle \geq 0 \), then it can be verified that \( \langle F(y), y - x \rangle \geq 0 \) holds. This follows from the fact that each \( F_i \) is strictly increasing in its own variable and independent of the others, which ensures that \( (F_i(y_i) - F_i(x_i))(y_i - x_i) \geq 0 \) for all \( i \), and hence
\[
\langle F(y) - F(x), y - x \rangle = \sum_{i=1}^3 (F_i(y_i) - F_i(x_i))(y_i - x_i) \geq 0.
\]
Thus, although \( F \) is neither symmetric nor Lipschitz, its structure yields pseudo-monotonicity, making it a valid and nontrivial test case for our Algorithm 1. We choose \( \Phi(x) = \sum_{i=1}^3 x_i \log x_i \). The first variant employs our Algorithm 1, while the second variant omits Algorithm 1 without the golden ratio technique. Both methods use adaptive step sizes with parameters \( \gamma = 0.8 \), \( \psi = \frac{1 + \sqrt{5}}{2} \), and \( \mu = 0.5 \), and initialize from \( x_0 = (0.3, 0.4, 0.3)^\top \). The stopping criterion is set to \( \|x_{k+1} - x_k\| < 10^{-4} \) or 800 iterations. 
\begin{figure}[ht]
\centering
\begin{minipage}{0.45\linewidth}
    \centering
    \includegraphics[width=\linewidth]{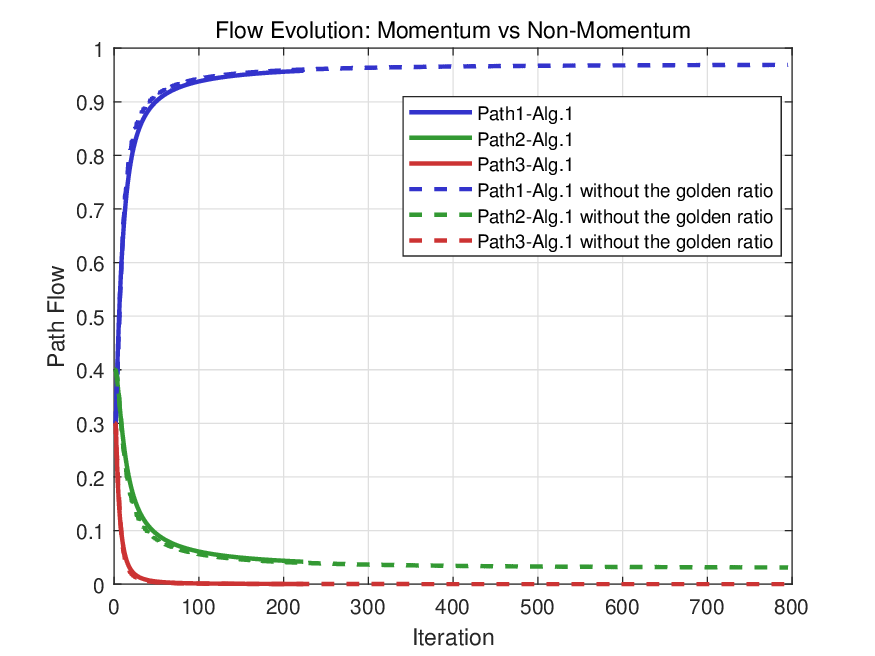}
    \label{fig:e1}
\end{minipage} \hfill
\begin{minipage}{0.45\linewidth}
    \centering
    \includegraphics[width=\linewidth]{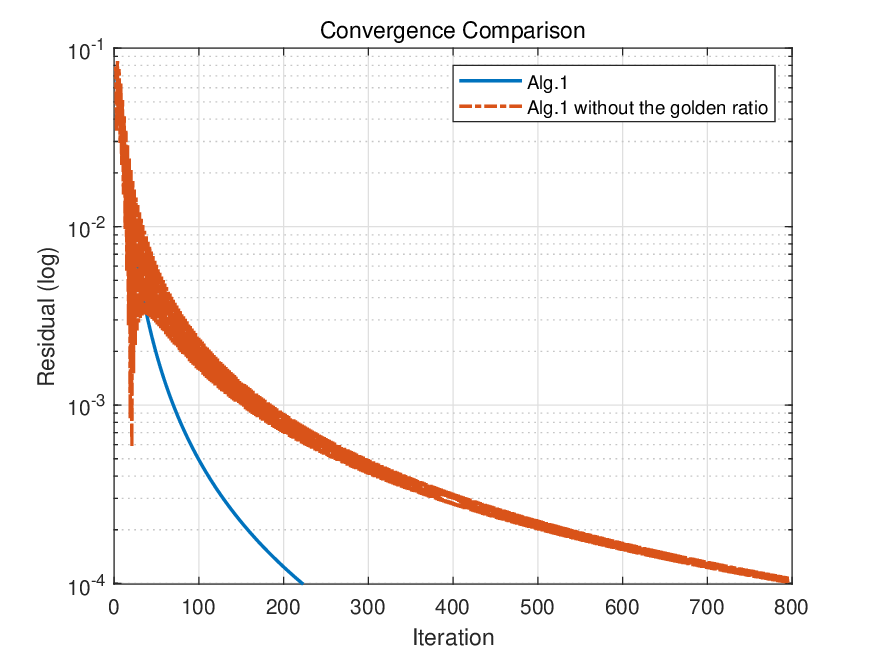}
    \label{fig:e2}
\end{minipage}
\caption{The numerical test results of Example 5.3.}
\label{fig:exp}
\end{figure}
The experimental results are shown in Figure~\ref{fig:exp}. The left panel illustrates the evolution of path flows over iterations. We observe that both methods converge to the same equilibrium allocation, where path 1 dominates due to its lowest base cost.  The right panel shows the residual errors \( \|x_{k+1} - x_k\| \) in logarithmic scale. The momentum method achieves rapid and smooth convergence within about 230 iterations, while the non-momentum variant converges slowly and exhibits oscillatory behavior. This confirms the effectiveness of using the golden ratio extrapolation to accelerate convergence in non-Lipschitz pseudo-monotone settings.

\section{Conclusion}
In this paper, we proposed and analyzed a novel forward-backward-forward (FBF) framework for solving quasimonotone variational inequalities. By extending the classical FBF dynamical system to handle quasimonotone operators, we established weak and strong convergence of the continuous trajectories under significantly relaxed assumptions. Notably, our analysis avoids reliance on strong pseudomonotonicity, sequential weak-to-weak continuity assumptions commonly required in the existing literature. In the discrete setting, we developed a new Bregman-type algorithm that employs golden-ratio-based extrapolation and a nonmonotone adaptive step-size strategy. This algorithm guarantees strong convergence under merely uniform continuity of the operator, thus significantly expanding the applicability of the method to broader classes of problems, including non-Lipschitz and infinite-dimensional settings.
In addition, we established the ergodic (time-average) convergence of the continuous trajectories based on the strong convergence result, offering a refined perspective on the long-term asymptotic behavior of the system.
Numerical experiments on infinite-dimensional and structured variational inequality problems confirm the theoretical properties and practical effectiveness of the proposed method, particularly in challenging scenarios where existing algorithms either fail to converge or require stronger conditions.

\vskip 2mm 
Although this work has yielded several promising results, it also leaves open some important questions that merit further investigation:
\vskip 2mm
(i) The existence and uniqueness of solutions to the proposed discrete dynamical system currently rely on the Lipschitz continuity of the operator \( F \). A natural question arises: can this assumption be weakened to mere uniform continuity? Moreover, while the convergence analysis has been conducted within finite-dimensional Hilbert spaces, it remains an open problem whether similar results hold in the infinite-dimensional setting.
\vskip 2mm
(ii) A fruitful direction is to investigate extensions of the proposed dynamical system to non-autonomous or stochastic settings. This includes analyzing the behavior of the trajectories when the operator \( F \) or the step-size rule depends explicitly on time, or when the system is subject to stochastic perturbations, as commonly encountered in real-world applications. Such generalizations would offer insights into the robustness, stability, and applicability of the method in more realistic scenarios.
\vskip 2mm
(iii) It is worth noting that our current ergodic (time-averaged) convergence results in the continuous-time setting rely on the strong convergence of the trajectory:
\[
x(t) \to x^*, \quad \text{as } t \to \infty.
\]
Whether the time-averaged weak convergence
\[
\bar{x}(T) := \frac{1}{T} \int_0^T x(t)\,dt \quad \rightharpoonup \quad x^*
\]
still holds under the weaker assumption that
\[
x(t) \rightharpoonup x^*
\]
without requiring strong convergence, remains an open question. Addressing this problem may require employing averaged variants of the Opial Lemma in continuous-time settings, along with a refined asymptotic analysis of the trajectory and the variational inequality structure. We leave this interesting question for future exploration.
\vskip 2mm 
These open questions not only highlight the depth and richness of the problem but also point toward fruitful directions for our future work, motivating us to further refine and generalize the proposed approach.

\section*{Declarations}

\begin{itemize}
\item Competing interests: The authors declare that they have no competing interests.
\item Availability of data and materials: The datasets generated during the current study are available from the corresponding author upon reasonable request.
\item Authors' contributions: All authors contributed equally to this work. All authors read and approved the final manuscript.
\item Funding: 
This research were supported by National Natural Science Foundation of China (No. 12261019), Natural Science Basic Research Program Project of Shaanxi Province (No. 2024JC-YBMS-019), the Fundamental Research Funds for the Central Universities and the Innovation Fund of Xidian University (No. YJSJ25009).
\end{itemize}

\end{document}